\documentclass[12pt, reqno]{amsart}

\usepackage{graphicx}        % standard LaTeX graphics tool

\usepackage{amsthm}

\usepackage{amsmath,amssymb}

\calclayout

\theoremstyle{plain}

\newtheorem{theorem}{Theorem}

\newtheorem{lemma}{Lemma}

\newtheorem{corollary}{Corollary} 

\theoremstyle{definition}

\newtheorem{remark}{Remark}

\newtoks\thehProclaim%
\newtheorem*{Proclaim}{\the\thehProclaim}

\theoremstyle{definition}%
\newtoks{\thehRemark}%
\newtheorem*{Remark}{\the\thehRemark}%

\begin{document}

\dedicatory{Dedicated with great pleasure and respect to Vladimir Maz'ya on the occasion of his 80th birthday}

\title[Fatou-Type Theorems and Boundary Value Problems]
{Fatou-Type Theorems and Boundary Value Problems for Elliptic Systems in the Upper Half-Space}

\author{Jos\'e Mar{\'\i}a Martell}
\address{Jos\'e Mar{\'\i}a Martell
\\
Instituto de Ciencias Matem\'aticas 
\\
CSIC-UAM-UC3M-UCM
\\
Consejo Superior de Investigaciones Cient{\'\i}ficas
\\
C/ Nicol\'as Cabrera, 13-15
\\
E-28049 Madrid, Spain} 
\email{chema.martell@icmat.es}

%\vspace{2mm}

\author{Dorina Mitrea}
\address{Dorina Mitrea
\\
Department of Mathematics
\\
University of Missouri
\\
Columbia, MO 65211, USA} 
\email{mitread@missouri.edu}

\author{Irina Mitrea}
\address{Irina Mitrea
\\
Department of Mathematics
\\
Temple University\!
\\
1805\,N.\,Broad\,Street
\\
Philadelphia, PA 19122, USA} 
\email{imitrea@temple.edu}

\author{Marius Mitrea}
\address{Marius Mitrea
\\
Department of Mathematics
\\
University of Missouri
\\
Columbia, MO 65211, USA} 
\email{mitream@missouri.edu}



\begin{abstract}
We survey recent progress in a program which to date has produced \cite{Madrid}-\cite{B-MMMM}, 
aimed at proving general Fatou-type results and establishing the well-posedness of a variety of
boundary value problems in the upper half-space ${\mathbb{R}}^n_{+}$ for second-order, homogeneous, 
constant complex coefficient, elliptic systems $L$, formulated in a manner that emphasizes pointwise 
nontangential boundary traces of the null-solutions of $L$ in ${\mathbb{R}}^n_{+}$.
\end{abstract}

\thanks{The first author acknowledges that the research leading to these results has received funding 
from the European Research Council under the European Union's Seventh Framework Programme (FP7/2007-2013)/ERC 
agreement no. 615112 HAPDEGMT. He also acknowledges financial support from the Spanish Ministry of Economy 
and Competitiveness, through the ``Severo Ochoa Programme for Centres of Excellence in R\&D'' (SEV-2015-0554). 
The second author has been supported in part by Simons Foundation grant $\#\,$426669, 
the third author has been supported in part by Simons Foundation grant $\#\,$318658,
while the fourth author has been supported in part by Simons Foundation grant $\#\,$281566. 
This work has been possible thanks to the support 
and hospitality of \textit{Temple University} (USA), \textit{University of Missouri} (USA), and 
\textit{ICMAT, Consejo Superior de Investigaciones Cient{\'\i}ficas} (Spain). The authors 
express their gratitude to these institutions.}

%\date{November 25, 2018}

\maketitle

\section{Introduction}
\label{S-1}

The topic of boundary value problems for elliptic operators in the upper half-space is a venerable subject 
which has received much attention throughout the years. While there is a wealth of results in which the 
smoothness of solutions and boundary data are measured on the scales of Sobolev, Besov, and Triebel-Lizorkin 
spaces (cf., e.g., \cite{ADNI}, \cite{ADNII}, \cite{FrRu}, \cite{Joh}, \cite{KMR1}, \cite{KMR2}, 
\cite{LionsMagenes}, \cite{Lop}, \cite{MaMiSh}, \cite{MazShap}, \cite{RuSi96}, \cite{Shap}, \cite{Sol1}, 
\cite{Sol2}, \cite{Taylor}, \cite{Tr83}, \cite{Tr95}, \cite{WRL} and the literature cited therein), the scenario 
in which the boundary traces are taken in a nontangential pointwise sense and the size of the solutions 
is measured using the nontangential maximal operator is considerably less understood. A notable exception 
is the case when the differential operator involved is the Laplacian, a situation dealt with at length 
in a number of monographs (cf., e.g., \cite{ABR}, \cite{GCRF85}, \cite{St70}, \cite{Stein93}, \cite{SW}). 
However, such undertakings always seem to employ rather specialized properties of harmonic functions, 
so new ideas are required when dealing with more general second-order elliptic systems in place of the 
Laplacian. In a sequence of recent works (cf. \cite{Madrid}, \cite{Holder-MMM}, \cite{H-MMMM}, 
\cite{K-MMMM}, 
\cite{S-MMMM}, \cite{BMO-MMMM}, \cite{SCGC}, \cite{B-MMMM}) the authors have systematically 
studied Fatou-type 
theorems and boundary value problems in the upper half-space for second-order elliptic homogeneous constant 
complex coefficient systems, in a formulation which emphasizes the nontangential pointwise behavior of 
the solutions on the boundary. 

The goal of this paper is to present for the first time a coherent, inclusive account of the 
progress registered so far in \cite{Madrid}-\cite{B-MMMM}. In \S\ref{S-2}, much attention is 
devoted to the topics of Poisson kernel and Fatou-type theorem. Complex problems typically 
call for a structured approach, and this is the path we follow vis-a-vis to the notion of 
Poisson kernel. Its original development is typically associated with the names of Agmon, Douglis, 
Nirenberg, Lopatinski\u{i}, Shapiro, Solonnikov, among others (cf. \cite{ADNI}-\cite{ADNII}, \cite{Lop}, 
\cite{Shap}-\cite{Sol2}), and here we further contribute to the study of Poisson kernels associated 
with second-order elliptic systems from the point of view of harmonic analysis. As regards the second 
topic of interest in \S\ref{S-2} mentioned earlier, recall that the trademark blueprint of a Fatou-type 
theorem is that certain size and integrability properties of a null-solution of an elliptic equation in 
a certain domain (often formulated in terms of the nontangential maximal operator) imply the a.e. existence 
of the pointwise nontangential boundary trace of the said function. Our Fatou-type theorems follow this 
design and are also quantitative in nature since the boundary trace does not just simply exist but encodes 
significant information regarding the size of the original function. 

In \S\ref{S-3} such results are used 
as tools for proving that a variety of boundary value problems for elliptic systems in the 
upper half-space 
are well-posed. In particular, here we monitor how the format of the problem changes
as the space of boundary data morphs from the Lebesgue scale $L^p$ with $1<p<\infty$, to the 
space of essentially bounded functions, to the space of functions of bounded mean oscillations 
and, further, to the space of H\"older continuous functions (or, more generally, the space of 
functions with sublinear growth). A significant number of results are new, and particular care is 
paid to understanding the extent to which the emerging theory is optimal. Along the way, a large 
number of relevant open problems are singled out for further study. 

We proceed to describe the class of systems employed in this work. 
Throughout, fix $n\in{\mathbb{N}}$ satisfying $n\geq 2$, along with $M\in{\mathbb{N}}$. 
Consider a second-order, homogeneous, $M\times M$ system, with constant complex coefficients, 
written (with the usual convention of summation
over repeated indices always in place, unless otherwise mentioned) as
\begin{equation}\label{L-def}
Lu:=\Bigl(\partial_r(a^{\alpha\beta}_{rs}\partial_s u_\beta)\Bigr)_{1\leq\alpha\leq M},
\end{equation}
when acting on $u=(u_\beta)_{1\leq\beta\leq M}$ whose components are distributions in an open 
subset of ${\mathbb{R}}^n$. Assume that $L$ is elliptic in the sense that there exists some 
$c\in(0,\infty)$ such that
\begin{equation}\label{L-ell.X}
\begin{array}{c}
{\rm Re}\,\bigl[a^{\alpha\beta}_{rs}\xi_r\xi_s\overline{\eta_\alpha}
\eta_\beta\,\bigr]\geq c|\xi|^2|\eta|^2\,\,\mbox{ for every}
\\[8pt]
\xi=(\xi_r)_{1\leq r\leq n}\in{\mathbb{R}}^n\,\,\mbox{ and }\,\,
\eta=(\eta_\alpha)_{1\leq\alpha\leq M}\in{\mathbb{C}}^M.
\end{array}
\end{equation}
Examples include scalar operators, such as the Laplacian $\Delta=\sum\limits_{j=1}^n\partial_j^2$ 
or, more generally, operators of the form ${\rm div}A\nabla$ with $A=(a_{rs})_{1\leq r,s\leq n}$ an 
$n\times n$ matrix with complex entries satisfying the ellipticity condition
\begin{equation}\label{YUjhv-753}
\inf_{\xi\in S^{n-1}}{\rm Re}\,\big[a_{rs}\xi_r\xi_s\bigr]>0,
\end{equation}
(where $S^{n-1}$ denotes the unit sphere in ${\mathbb{R}}^n$), as well as complex
versions of the Lam\'e system of elasticity
\begin{equation}\label{TYd-YG-76g}
\begin{array}{c}
L:=\mu\Delta+(\lambda+\mu)\nabla{\rm div}\,\,\text{ where the Lam\'e moduli }\,\,\lambda,\mu\in{\mathbb{C}}
\\[6pt]
\text{satisfy }\,\,{\rm Re}\,\mu>0\,\,\mbox{ and }\,\,{\rm Re}\,(2\mu+\lambda)>0.
\end{array}
\end{equation}
The last condition above is equivalent to the demand that the Lam\'e system \eqref{TYd-YG-76g} 
is Legendre-Hadamard elliptic (in the sense of \eqref{L-ell.X}). While the Lam\'e system is symmetric, 
we stress that the results in this paper require no symmetry for the systems involved. 

We shall work in the upper half-space
\begin{equation}\label{RRR-UpHs}
{\mathbb{R}}^{n}_{+}:=\big\{x=(x',x_n)\in
{\mathbb{R}}^{n}={\mathbb{R}}^{n-1}\times{\mathbb{R}}:\,x_n>0\big\}
\end{equation}
whose topological boundary we shall henceforth identify with the horizontal hyperplane ${\mathbb{R}}^{n-1}$ via 
$\partial{\mathbb{R}}^{n}_{+}\ni(x',0)\equiv x'\in{\mathbb{R}}^{n-1}$. The origin in ${\mathbb{R}}^{n-1}$ 
is denoted by $0'$, and we agree to let $B_{n-1}(x',r):=\{y'\in{\mathbb{R}}^{n-1}:\,|x'-y'|<r\}$ stand for 
the $(n-1)$-dimensional ball centered at $x'\in{\mathbb{R}}^{n-1}$ and of radius $r>0$.  
We shall also let ${\mathbb{N}}_0$ stand for the collection of all non-negative integers. 
Finally, we will adopt the standard custom of allowing the letter $C$ to denote constants 
which may vary from one occurrence to another.

\section{Poisson Kernels and General Fatou-Type Results}
\label{S-2}

Poisson kernels for elliptic operators in a half-space have a long history 
(see, e.g., \cite{ADNI}, \cite{ADNII}, \cite{Lop}, \cite{Shap}, \cite{Sol1}, \cite{Sol2}). 
In the theorem below we single out the most essential features which identify these objects uniquely. 

\begin{theorem}\label{thm:Poisson}
Let $L$ be an $M\times M$ homogeneous constant complex coefficient elliptic second-order system in ${\mathbb{R}}^n$. 
Then there exists a matrix-valued function
\begin{equation}\label{PP-OO}
P^L=\big(P^L_{\alpha\beta}\big)_{1\leq\alpha,\beta\leq M}:
\mathbb{R}^{n-1}\longrightarrow\mathbb{C}^{M\times M}
\end{equation}
{\rm (}called the Poisson kernel for $L$ in $\mathbb{R}^{n}_{+}${\rm )}
satisfying the following properties:
\begin{list}{$(\theenumi)$}{\usecounter{enumi}\leftmargin=.8cm
\labelwidth=.8cm\itemsep=0.2cm\topsep=.1cm
\renewcommand{\theenumi}{\alph{enumi}}}
\item There exists $C\in(0,\infty)$ such that
\begin{equation}\label{eq:IG6gy}
|P^L(x')|\leq\frac{C}{(1+|x'|^2)^{\frac{n}2}}\quad\mbox{for each }\,\,x'\in\mathbb{R}^{n-1}.
\end{equation}
\item The function $P^L$ is Lebesgue measurable and
\begin{equation}\label{eq:IG6gy.2}
\int_{\mathbb{R}^{n-1}}P^L(x')\,dx'=I_{M\times M},
\end{equation}
where $I_{M\times M}$ denotes the $M\times M$ identity matrix. 
\item If one sets
\begin{equation}\label{eq:Gvav7g5}
\begin{array}{c}
K^L(x',t):=P^L_t(x')=t^{1-n}P^L(x'/t)
\\[6pt]
\mbox{for each }\,\,x'\in\mathbb{R}^{n-1}\,\,\,\mbox{ and }\,\,t>0,
\end{array}
\end{equation}
then the $\mathbb{C}^{M\times M}$-valued function $K^L$ satisfies 
{\rm (}with $L$ acting on the columns of $K^L$ in the sense of distributions{\rm )}
\begin{equation}\label{uahgab-UBVCX}
LK^L=0\cdot I_{M\times M}\,\,\text{ in }\,\,\big[{\mathcal{D}}'(\mathbb{R}^{n}_{+})\big]^{M\times M}.
\end{equation}
\item The Poisson kernel $P^L$ is unique in the class of $\mathbb{C}^{M\times M}$-valued functions 
defined in ${\mathbb{R}}^{n-1}$ and satisfying $(a)$-$(c)$ above.
\end{list}
\end{theorem}

Concerning Theorem~\ref{thm:Poisson}, we note that the existence part follows from the classical 
work of S.\,Agmon, A.\,Douglis, and L.\,Nirenberg in \cite{ADNII} (cf. also \cite{Lop}, 
\cite{Shap}-\cite{Sol2}). The uniqueness property has been recently proved in \cite{K-MMMM}.

The Poisson kernel introduced above is the basic tool used to construct solutions for 
the Dirichlet problem for the system $L$ in the upper half-space. This is most 
apparent from Theorem~\ref{thm:Poisson.II} stated a little further below. 
For now, we proceed to define the nontangential maximal operator and the nontangential 
boundary trace. Specifically, having fixed some aperture parameter $\kappa>0$, at each 
point $x'\in\partial{\mathbb{R}}^{n}_{+}\equiv{\mathbb{R}}^{n-1}$ we define the conical
nontangential approach region with vertex at $x'$ as
\begin{equation}\label{NT-1}
\Gamma_\kappa(x'):=\big\{y=(y',t)\in{\mathbb{R}}^{n}_{+}:\,|x'-y'|<\kappa\,t\big\}.
\end{equation}
Given a continuous vector-valued function $u:{\mathbb{R}}^{n}_{+}\to{\mathbb{C}}^M$,
we then define the nontangential maximal operator acting on $u$ by setting 
\begin{equation}\label{NT-Fct}
\big({\mathcal{N}}_\kappa u\big)(x'):=\sup\big\{|u(y)|:\,y\in\Gamma_\kappa(x')\big\},\qquad
x'\in{\mathbb{R}}^{n-1}.
\end{equation}
We shall also need a version of the nontangential maximal operator in which 
the supremum is now taken over cones truncated near the vertex. Specifically, 
given a continuous vector-valued function $u:{\mathbb{R}}^{n}_{+}\to{\mathbb{C}}^M$,
for each $\varepsilon>0$ define 
\begin{equation}\label{NT-Fct-EP}
\big({\mathcal{N}}^{(\varepsilon)}_\kappa u\big)(x'):=\sup\big\{|u(y)|:\,
y=(y',t)\in\Gamma_\kappa(x')\,\text{ with }\,t>\varepsilon\big\}
\end{equation}
at each $x'\in{\mathbb{R}}^{n-1}$.
Whenever meaningful, the $\kappa$-nontangential pointwise boundary trace of 
a continuous vector-valued function $u:{\mathbb{R}}^{n}_{+}\to{\mathbb{C}}^M$ is given by
\begin{equation}\label{nkc-EE-2}
\Big(u\big|^{{}^{\kappa-{\rm n.t.}}}_{\partial{\mathbb{R}}^{n}_{+}}\Big)(x')
:=\lim_{\Gamma_{\kappa}(x')\ni y\to (x',0)}u(y)
\,\,\mbox{ for }\,\,x'\in\partial{\mathbb{R}}^{n}_{+}\equiv{\mathbb{R}}^{n-1}.
\end{equation}
It is then clear from definitions that for any continuous vector-valued function $u:{\mathbb{R}}^{n}_{+}\to{\mathbb{C}}^M$
and any $\varepsilon,\kappa>0$ we have 
\begin{equation}\label{NT-Fct-EP.anB}
\begin{array}{c}
{\mathcal{N}}^{(\varepsilon)}_\kappa u,\,\,\,{\mathcal{N}}_\kappa u
\,\,\text{ are lower semicontinuous}, 
\\[6pt]
0\leq{\mathcal{N}}^{(\varepsilon)}_\kappa u\leq{\mathcal{N}}_\kappa u
\,\,\text{ on }\,\,\partial{\mathbb{R}}^{n}_{+}\equiv{\mathbb{R}}^{n-1}.
\end{array}
\end{equation}
In addition, for each such function $u$ we have
\begin{equation}\label{6543}
\|u\|_{[L^\infty(\mathbb{R}^n_{+})]^M}=\|\mathcal{N}_\kappa u\|_{L^\infty(\mathbb{R}^{n-1})}.
\end{equation}
Finally, whenever the nontangential boundary trace exists, we have
\begin{equation}\label{nkc-EE-4}
\begin{array}{c}
u\big|^{{}^{\kappa-{\rm n.t.}}}_{\partial{\mathbb{R}}^n_{+}}\,\,\text{ is a Lebesgue measurable function} 
\\[6pt]
\text{and }\,\,\left|u\big|^{{}^{\kappa-{\rm n.t.}}}_{\partial{\mathbb{R}}^n_{+}}\right|\leq{\mathcal{N}}_\kappa u
\,\,\text{ on }\,\,\partial{\mathbb{R}}^{n}_{+}\equiv{\mathbb{R}}^{n-1}.
\end{array}
\end{equation}

Prior to stating our next result, we make some comments further clarifying notation and terminology.
Throughout, we agree to denote by $\mathcal{M}$ the Hardy-Littlewood maximal operator on $\mathbb{R}^{n-1}$.
This acts on each vector-valued function $f$ with components in $L^1_{\rm loc}({\mathbb{R}}^{n-1})$ according to
\begin{equation}\label{MMax}
\big(\mathcal{M}f\big)(x'):=\sup_{Q\ni x'}\frac{1}{{\mathcal{L}}^{n-1}(Q)}
\int_Q|f|\,d{\mathcal{L}}^{n-1},\qquad\forall\,x'\in\mathbb{R}^{n-1},
\end{equation}
where the supremum runs over all cubes $Q$ in $\mathbb{R}^{n-1}$ containing $x'$,
and where ${\mathcal{L}}^{n-1}$ denotes the $(n-1)$-dimensional Lebesgue measure 
in ${\mathbb{R}}^{n-1}$. 

Next, pick some integrability exponent $q\in(1,\infty)$ (whose actual choice is ultimately immaterial), 
and fix an arbitrary $p\in\big(\tfrac{n-1}{n}\,,\,1\big]$. Recall that a Lebesgue measurable function 
$a:\mathbb{R}^{n-1}\rightarrow\mathbb{C}$ is said to be an $(p,q)$-atom if 
for some cube $Q\subset\mathbb{R}^{n-1}$ one has
\begin{equation}\label{defi-atom}
{\rm supp}\,a\subset Q,\quad\|a\|_{L^q(\mathbb{R}^{n-1})}\leq{\mathcal{L}}^{n-1}(Q)^{1/q-1/p}, 
\quad\int_{\mathbb{R}^{n-1}}a\,d{\mathcal{L}}^{n-1}=0.
\end{equation}
One may then define the Hardy space $H^p(\mathbb{R}^{n-1})$ as the collection of all tempered 
distributions $f\in{\mathcal{S}}'(\mathbb{R}^{n-1})$ which may be written as 
\begin{equation}\label{66tt}
f=\sum_{j\in{\mathbb{N}}}\lambda_j\,a_j\,\,\text{ in }\,\,{\mathcal{S}}'(\mathbb{R}^{n-1})
\end{equation}
for some sequence $\{a_j\}_{j\in{\mathbb{N}}}$ of $(p,q)$-atoms and a sequence  
$\{\lambda_j\}_{j\in{\mathbb{N}}}\in\ell^p$. For each $f\in H^p(\mathbb{R}^{n-1})$ we then set 
$\|f\|_{H^p(\mathbb{R}^{n-1})}:=\inf\Big(\sum_{j\in{\mathbb{N}}}|\lambda_j|^p\Big)^{1/p}$ with the 
infimum taken over all atomic decompositions of $f$ as $\sum_{j\in{\mathbb{N}}}\lambda_j\,a_j$.

In relation to this we wish to make three comments. First, the very definition of the quasi-norm 
$\|\cdot\|_{H^p(\mathbb{R}^{n-1})}$ implies that whenever $f\in H^p(\mathbb{R}^{n-1})$ is written 
as in \eqref{66tt} then the series actually converges in $H^p(\mathbb{R}^{n-1})$. Second, from the 
definition of $\|\cdot\|_{H^p(\mathbb{R}^{n-1})}$ we also see that each $f\in H^p(\mathbb{R}^{n-1})$ 
has a quasi-optimal atomic decomposition, i.e., $f$ may be written as in \eqref{66tt} with 
\begin{equation}\label{66tt.222}
\frac{1}{2}\Big(\sum_{j\in{\mathbb{N}}}|\lambda_j|^p\Big)^{1/p}
\leq\|f\|_{H^p(\mathbb{R}^{n-1})}\leq\Big(\sum_{j\in{\mathbb{N}}}|\lambda_j|^p\Big)^{1/p}.
\end{equation}
Third, consider the vector case, i.e., the space $\big[H^p({\mathbb{R}}^{n-1})\big]^M$. 
In such a setting we find it convenient to work with $\mathbb{C}^M$-valued $(p,q)$-atoms. Specifically, 
these are functions 
\begin{equation}\label{defi-atom-CM}
\begin{array}{c}
a\in\big[L^q(\mathbb{R}^{n-1})\big]^M\,\,\text{ such that for some cube $Q\subset\mathbb{R}^{n-1}$ one has}
\\[6pt]
{\rm supp}\,a\subset Q,\quad\|a\|_{[L^q(\mathbb{R}^{n-1})]^M}\leq{\mathcal{L}}^{n-1}(Q)^{1/q-1/p}, 
\\[6pt]
\text{and }\displaystyle\int_{\mathbb{R}^{n-1}}a\,d{\mathcal{L}}^{n-1}=0\in{\mathbb{C}}^M.
\end{array}
\end{equation}
Suppose now that some $f=(f_\beta)_{1\leq\beta\leq M}\in\big[H^p({\mathbb{R}}^{n-1})\big]^M$ has been given. 
Then each $f_\beta$ has an atomic decomposition $f_\beta=\sum_{j=1}^\infty\lambda_{\beta j}a_{\beta j}$ 
(no summation on $\beta$ here) where each $a_{\beta j}$ is a $(p,q)$-atom and 
$\{\lambda_{\beta j}\}_{j\in{\mathbb{N}}}\in\ell^p$, which is quasi-optimal,
hence 
\begin{align}\label{exist:u-123.Wer.1aaa}
\|f_\beta\|_{H^p({\mathbb{R}}^{n-1})}\approx\big(\sum_{j=1}^\infty|\lambda_{\beta j}|^p\big)^{1/p}
\,\,\text{ for each }\,\,\beta\in\{1,\dots,M\}.
\end{align}
Using the Kronecker symbol formalism, introduce 
${\mathbf{e}}_\beta:=(\delta_{\gamma\beta})_{1\leq\gamma\leq M}\in{\mathbb{C}}^M$
for each index $\beta\in\{1,\dots,M\}$, then write 
\begin{align}\label{exist:u-123.Wer.1}
f=\sum_{\beta=1}^Mf_\beta{\mathbf{e}}_\beta
=\sum_{\beta=1}^M\sum_{j=1}^\infty\lambda_{\beta j}a_{\beta j}{\mathbf{e}}_\beta
=\sum_{\beta=1}^M\sum_{j=1}^\infty\lambda_{\beta j}A_{\beta j}
\end{align}
with convergence in $\big[H^p({\mathbb{R}}^{n-1})\big]^M$, where 
\begin{align}\label{exist:u-123.Wer.2}
A_{\beta j}:=a_{\beta j}{\mathbf{e}}_\beta\,\,\text{ for each }\,\,
\beta\in\{1,\dots,M\}\,\,\text{ and }\,\,j\in{\mathbb{N}}
\end{align}
are ${\mathbb{C}}^M$-valued functions as in \eqref{defi-atom-CM}, hence $\mathbb{C}^M$-valued $(p,q)$-atoms. 
If we then relabel the sequences $\big\{A_{\beta j}\big\}_{\substack{1\leq\beta\leq M\\ j\in{\mathbb{N}}}}$ 
and $\big\{\lambda_{\beta j}\big\}_{\substack{1\leq\beta\leq M\\ j\in{\mathbb{N}}}}$ simply as
$\big\{a_j\big\}_{j\in{\mathbb{N}}}$ and $\big\{\lambda_j\big\}_{j\in{\mathbb{N}}}$, respectively, 
we may re-cast  \eqref{exist:u-123.Wer.1aaa}-\eqref{exist:u-123.Wer.2} as 
\begin{equation}\label{exist:u-123.Wer.3}
\begin{array}{c}
f=\sum_{j=1}^\infty\lambda_ja_j\,\,\text{ with convergence in }\,\,\big[H^p({\mathbb{R}}^{n-1})\big]^M,
\\[6pt]
\text{where each $a_j$ is a $\mathbb{C}^M$-valued $(p,q)$-atom (cf. \eqref{defi-atom-CM})},
\\[6pt]
\text{and }\,\,\|f\|_{[H^p({\mathbb{R}}^{n-1})]^M}\approx\big(\sum_{j=1}^\infty|\lambda_j|^p\big)^{1/p}.
\end{array}
\end{equation}

Since the Poisson kernel $P^L$ and the kernel function $K^L$ from Theorem~\ref{thm:Poisson} 
are of fundamental importance to the work described in this paper, a more in-depth analysis 
of their main properties is in order. Before stating our theorem addressing this analysis, 
for each real number $m$ we agree to denote 
by $L^1\big({\mathbb{R}}^{n-1}\,,\,\tfrac{dx'}{1+|x'|^{m}}\big)$ the space of 
${\mathbb{C}}$-valued Lebesgue measurable functions which are absolutely integrable in 
${\mathbb{R}}^{n-1}$ with respect to the weighted Lebesgue measure $\tfrac{dx'}{1+|x'|^{m}}$. 

\begin{theorem}\label{thm:Poisson.II}
Let $L$ be an $M\times M$ homogeneous constant complex coefficient elliptic second-order system in ${\mathbb{R}}^n$. 
Then the Agmon-Douglis-Nirenberg Poisson kernel $P^L$ and the kernel function $K^L$ from Theorem~\ref{thm:Poisson} 
satisfy the following properties:
\begin{list}{$(\theenumi)$}{\usecounter{enumi}\leftmargin=.8cm
\labelwidth=.8cm\itemsep=0.2cm\topsep=.1cm
\renewcommand{\theenumi}{\alph{enumi}}}
\item The function $P^L$ belongs to $\big[{\mathcal{C}}^\infty(\mathbb{R}^{n-1})\big]^{M\times M}$ and 
satisfies the following non-dege\-neracy property: 
\begin{equation}\label{eq:IG6gy.2-Lambda}
\parbox{6.90cm}{for each $a\in{\mathbb{C}}^M\setminus\{0\}$ one may find some $\lambda>0$ such that 
$\int_{S^{n-2}}\big|P^L(\lambda\omega)a\big|\,d\omega>0$.}
\end{equation}
One may extend $K^L$ to a function belonging to 
$\big[{\mathcal{C}}^\infty\big(\overline{{\mathbb{R}}^n_{+}}\setminus\{0\}\big)\big]^{M\times M}$.
Consequently, formula \eqref{uahgab-UBVCX} also holds in a pointwise sense in ${\mathbb{R}}^n_{+}$. 
Moreover, there exists some constant $C\in(0,\infty)$ such that 
\begin{equation}\label{Uddcv}
|K^L(x',t)|\leq Ct/(t^2+|x'|^2)^{n/2}\,\,\text{ for each }
\,\,(x',t)\in{\mathbb{R}}^n_{+},
\end{equation}
an estimate which further implies $K^L(x',0)=0$ for each 
$x'\in{\mathbb{R}}^{n-1}\setminus\{0'\}$. In addition, one has 
$\int_{\mathbb{R}^{n-1}}K^L(x'-y',t)\,dy'=I_{M\times M}$ for all $(x',t)\in{\mathbb{R}}^n_{+}$, 
as well as $K^L(\lambda x)=\lambda^{1-n}K^L(x)$ for all $x\in{\mathbb{R}}^n_{+}$ and $\lambda>0$. 
In particular, for each multi-index $\alpha\in{\mathbb{N}}_0^n$ there exists 
$C_\alpha\in(0,\infty)$ with the property that
\begin{equation}\label{eq:Kest}
\big|(\partial^\alpha K^L)(x)\big|\leq C_\alpha\,|x|^{1-n-|\alpha|},\qquad
\forall\,x\in{\overline{{\mathbb{R}}^n_{+}}}\setminus\{0\}.
\end{equation}
\item The following semi-group property holds:
\begin{equation}\label{u7tffa}
P^L_{t_0+t_1}=P^L_{t_0}\ast P^L_{t_1}\,\,\text{ for all }\,\,t_0,t_1>0.
\end{equation}
\item Given a Lebesgue measurable function $f=(f_\beta)_{1\leq\beta\leq M}:\mathbb{R}^{n-1}\rightarrow\mathbb{C}^M$ 
satisfying
\begin{equation}\label{exist:f}
\int_{\mathbb{R}^{n-1}}\frac{|f(x')|}{1+|x'|^n}\,dx'<\infty,
\end{equation}
at each point $(x',t)\in{\mathbb{R}}^n_{+}$ set
\begin{align}\label{exist:u}
u(x',t) &:=(P^L_t\ast f)(x')
\nonumber\\[6pt]
&:=\Bigg(\int_{{\mathbb{R}}^{n-1}}t^{1-n}P^L_{\alpha\beta}\big((x'-y')/t\big)f_\beta(y')\,dy'\Bigg)_{1\leq\alpha\leq M}.
\end{align}
Then $u:\mathbb{R}^n_{+}\to\mathbb{C}^M$ is meaningfully defined via an absolutely convergent integral,
satisfies {\rm (}for each given aperture parameter $\kappa>0${\rm )}
\begin{equation}\label{exist:u2}
\begin{array}{c}
u\in\big[\mathcal{C}^\infty(\mathbb{R}^n_{+})\big]^M,\quad
Lu=0\,\,\mbox{ in }\,\,\mathbb{R}^{n}_{+},
\\[8pt]
\text{and }\,\,
u\big|_{\partial\mathbb{R}^{n}_{+}}^{{}^{\kappa-{\rm n.t.}}}=f
\,\,\mbox{ at each Lebesgue point of $f$}
\\[6pt]
{\rm (}\text{hence, in particular, at ${\mathcal{L}}^{n-1}$-a.e. point in $\mathbb{R}^{n-1}$}{\rm )},
\end{array}
\end{equation}
and there exists a constant $C=C(L,\kappa)\in(0,\infty)$ with the property that
\begin{equation}\label{exist:Nu-Mf}
\big(\mathcal{N}_\kappa u\big)(x')\leq C\,\mathcal{M} f(x')\,\,\text{ for each }\,\,\,x'\in\mathbb{R}^{n-1}.
\end{equation}
Also, if $L=\Delta$, the Laplacian in ${\mathbb{R}}^n$, then the opposite inequality in \eqref{exist:Nu-Mf}
is true as well. Furthermore, the following unrestricted convergence result holds 
\begin{equation}\label{exist:Nu-Mf-LIM}
\begin{array}{c}
\lim\limits_{{\mathbb{R}}^n_{+}\ni x\to (x'_0,0)}u(x)=f(x'_0)
\\[6pt]
\text{if $x'_0\in{\mathbb{R}}^{n-1}$ is a continuity point for $f$}.
\end{array}
\end{equation}
In other words, $u$ given by \eqref{exist:u} extends by continuity to 
${\mathbb{R}}^n_{+}\cup\{(x'_0,0)\}$ whenever $x'_0\in{\mathbb{R}}^{n-1}$ is a continuity point for $f$.
In particular, 
\begin{equation}\label{exist:Nu-Mf-LIM.222}
\parbox{10.00cm}{whenever 
$f\in\Big[L^1\Big({\mathbb{R}}^{n-1}\,,\,\frac{dx'}{1+|x'|^{n}}\Big)\cap{\mathcal{C}}^0({\mathbb{R}}^{n-1})\Big]^M$ 
and $u$ is as in \eqref{exist:u}, then $u$ extends uniquely to a function in 
$\big[\mathcal{C}^0(\overline{\mathbb{R}^n_{+}})\big]^M$.}
\end{equation} 
\item For each $p\in\big(\tfrac{n-1}{n}\,,\,1\big]$ and each $\alpha\in{\mathbb{N}}_0^n$ with $|\alpha|>0$,  
the kernel function $K^L$ satisfies
\begin{equation}\label{grefr}
\begin{array}{c}
(\partial^\alpha K^L)(x'-\cdot,t)\in\big[H^p(\mathbb{R}^{n-1})\big]^{M\times M}
\,\text{ for all }\,(x',t)\in{\mathbb{R}}^n_{+},\,\text{ and}
\\[6pt]
\sup\limits_{(x',t)\in{\mathbb{R}}^n_{+}}
\big\|t^{|\alpha|-(n-1)(\frac1p-1)}(\partial^\alpha K^L)(x'-\cdot,t)\big\|_{[H^p(\mathbb{R}^{n-1})]^{M\times M}}<\infty.
\end{array}
\end{equation}
In fact, for each $p\in\big(\tfrac{n-1}{n}\,,\,1\big]$, each $q\in[1,\infty]$ with $q>p$, each 
$\alpha\in{\mathbb{N}}_0^n$ with $|\alpha|>0$, and each $(x',t)\in{\mathbb{R}}^n_{+}$, 
the function 
\begin{equation}\label{lY54dvb}
m^\alpha_{x',t}:=t^{|\alpha|-(n-1)(\frac1p-1)}(\partial^\alpha K^L)(x'-\cdot,t)
\end{equation}
is, up to multiplication by some fixed constant $C\in(0,\infty)$ 
{\rm (}which depends exclusively on $L,p,q,n,\alpha${\rm )}, 
a ${\mathbb{C}}^{M\times M}$-valued $L^q$-normalized molecule relative to the ball 
$B_{n-1}(x',t)$ for the Hardy space $\big[H^p(\mathbb{R}^{n-1})\big]^{M\times M}$. 
More precisely, 
\begin{equation}\label{RWWQD-bb}
\int_{\mathbb{R}^{n-1}}m^\alpha_{x',t}\,d{\mathcal{L}}^{n-1}=0\cdot I_{M\times M},
\end{equation}
and there exists $C\in(0,\infty)$ such that one has
\begin{equation}\label{RWWQD-cc}
\big\|m^\alpha_{x',t}\big\|_{[L^q(B_{n-1}(x',t))]^{M\times M}}
\leq C{\mathcal{L}}^{n-1}\big(B_{n-1}(x',t)\big)^{\frac1q-\frac1p},
\end{equation}
and, using the abbreviation $\varepsilon:=|\alpha|/(n-1)$, for each $k\in{\mathbb{N}}$ one also has
\begin{align}\label{RWWQD-cc.222}
&\hskip -0.30in
\big\|m^\alpha_{x',t}\big\|_{[L^q(B_{n-1}(x',2^{k}t)\setminus B_{n-1}(x',2^{k-1}t))]^{M\times M}}
\nonumber\\[6pt]
&\hskip 0.80in
\leq C2^{k(n-1)\big(\frac1q-1-\varepsilon\big)}{\mathcal{L}}^{n-1}\big(B_{n-1}(x',t)\big)^{\frac1q-\frac1p}.
\end{align}

\item Given any $\alpha\in{\mathbb{N}}_0^n$, for each fixed $t>0$ the function 
$(\partial^\alpha K^L)(\cdot,t)$ belongs to the H\"older space 
$\big[{\mathcal{C}}^\theta({\mathbb{R}}^{n-1})\big]^{M\times M}$ for each 
exponent $\theta\in(0,1)$. As a consequence of this, given an arbitrary 
$f=(f_\beta)_{1\leq\beta\leq M}\in\big[H^p(\mathbb{R}^{n-1})\big]^{M}$ with 
$p\in\big(\tfrac{n-1}{n}\,,\,1\big]$, one may meaningfully define
\begin{align}\label{exist:u-123}
u(x',t) &:=(P^L_t\ast f)(x')
\\[6pt]
&:=\Bigg\{\Big\langle f_\beta\,,\,\big[K^L_{\alpha\beta}(x'-\cdot,t)\big]\Big\rangle\Bigg\}_{1\leq\alpha\leq M}
\,\text{ for }\,(x',t)\in{\mathbb{R}}^n_{+},
\nonumber
\end{align}
where $\langle\cdot,\cdot\rangle$ is the pairing between distributions belonging to the Hardy space 
$H^p(\mathbb{R}^{n-1})$ and equivalence classes {\rm (}modulo constants{\rm )} of functions belonging to 
the homogeneous H\"older space $\dot{\mathcal{C}}^{(n-1)(1/p-1)}({\mathbb{R}}^{n-1})$ if $p<1$, and 
to ${\rm BMO}({\mathbb{R}}^{n-1})$ if $p=1$ {\rm (}cf., e.g., \cite[Theorem~5.30, p.307]{GCRF85}{\rm )}. Then 
\begin{equation}\label{exist:u2-Hp}
\begin{array}{c}
u\in\big[\mathcal{C}^\infty(\mathbb{R}^n_{+})\big]^M,\quad
Lu=0\,\,\mbox{ in }\,\,\mathbb{R}^{n}_{+},
\\[8pt]
\text{and }\,\,u\big|_{\partial\mathbb{R}^{n}_{+}}^{{}^{\kappa-{\rm n.t.}}}
\,\,\mbox{ exists ${\mathcal{L}}^{n-1}$-a.e. in $\mathbb{R}^{n-1}$}.
\end{array}
\end{equation}
Moreover, there exists a constant $C=C(L,\kappa,p)\in(0,\infty)$ with the property that
\begin{equation}\label{exist:Nu-Mf-Hp}
\begin{array}{c}
\big\|\mathcal{N}_\kappa u\big\|_{L^p(\mathbb{R}^{n-1})}
\leq C\|f\|_{[H^p(\mathbb{R}^{n-1})]^M}\,\,\text{ whenever}
\\[6pt]
f\in\big[H^p(\mathbb{R}^{n-1})\big]^{M}\,\,\text{ and $u$ is as in \eqref{exist:u-123}.}
\end{array}
\end{equation}
\end{list}
\end{theorem}

We wish to note that, in sharp contrast with \eqref{exist:u2}, in the context of \eqref{exist:u2-Hp} we no 
longer expect the nontangential pointwise trace $u\big|_{\partial\mathbb{R}^{n}_{+}}^{{}^{\kappa-{\rm n.t.}}}$
to be directly related to the (generally speaking tempered distribution) $f\in\big[H^p(\mathbb{R}^{n-1})\big]^M$.
For example, in \eqref{FCT-TR.nnn}-\eqref{rt-vba-jg} we present an example in which the said 
trace vanishes at ${\mathcal{L}}^{n-1}$-a.e. point in $\mathbb{R}^{n-1}$ even though $f\not=0$. 

\begin{proof}[Proof of Theorem~\ref{thm:Poisson.II}]
For items {\it (a)}-{\it (c)} see \cite{K-MMMM}, \cite{S-MMMM}, \cite{SCGC}, \cite{B-MMMM}. 
To deal with the claims in item {\it (d)}, fix $p,q,\alpha,x',t$ as in the statement. Then 
\begin{align}\label{est-theta-vanish}
\int_{\mathbb{R}^{n-1}}m^\alpha_{x',t}(y')\,dy'
=t^{|\alpha|-(n-1)(\frac{1}{p}-1)}\,\partial^\alpha\int_{\mathbb{R}^{n-1}}K^L(x'-y',t)\,dy'=0
\end{align}
since $\int_{\mathbb{R}^{n-1}}K^L(x'-y',t)\,dy'=I_{M\times M}$ and $|\alpha|>0$. 
This proves \eqref{RWWQD-bb}. Also, based on \eqref{eq:Kest} we may estimate 
\begin{align}\label{RWWQD-dd}
\int_{B_{n-1}(x',t)}\big|m^\alpha_{x',t}(y')\big|^q\,dy'
&\leq C\,\int_{B_{n-1}(x',t)}\frac{t^{q|\alpha|-q(n-1)(\frac{1}{p}-1)}}{(t+|x'-y'|)^{q(n-1+|\alpha|)}}\,dy'
\nonumber\\[4pt]
&\leq C\,\int_{B_{n-1}(x',t)}\frac{t^{q|\alpha|-q(n-1)(\frac{1}{p}-1)}}{t^{q(n-1+|\alpha|)}}\,dy'
\nonumber\\[4pt]
&=C\Big[{\mathcal{L}}^{n-1}\big(B_{n-1}(x',t)\big)^{\frac1q-\frac1p}\Big]^q,
\end{align}
and, if $\varepsilon:=|\alpha|/(n-1)$, for every $k\in{\mathbb{N}}$ we may write 
\begin{align}\label{RWWQD-ee}
\int_{B_{n-1}(x',2^kt)\setminus B_{n-1}(x',2^{k-1}t)} &\big|m^\alpha_{x',t}(y')\big|^q\,dy'
\nonumber\\[4pt]
&\hskip -0.60in
\leq C\,\int_{2^{k-1}\,t<|x'-y'|<2^k\,t}\frac{t^{q|\alpha|-q(n-1)(\frac{1}{p}-1)}}{(t+|x'-y'|)^{q(n-1+|\alpha|)}}\,dy'
\nonumber\\[4pt]
&\hskip -0.60in
\leq C\int_{B_{n-1}(x',2^kt)}\frac{t^{q|\alpha|-q(n-1)(\frac{1}{p}-1)}}{(2^k\,t)^{q(n-1+|\alpha|)}}\,dy'
\nonumber\\[4pt]
&\hskip -0.60in
=C\Big[2^{k(n-1)\big(\frac1q-1-\varepsilon\big)}{\mathcal{L}}^{n-1}\big(B_{n-1}(x',t)\big)^{\frac1q-\frac1p}\Big]^q,
\end{align}
for some constant $C\in(0,\infty)$ independent of $k$, $x'$, and $t$. From these, 
the estimates claimed in \eqref{RWWQD-cc}, \eqref{RWWQD-cc.222} readily follow. 
Going further, the first claim in item {\it (e)} is a consequence of the fact that, 
as seen from \eqref{eq:Kest}, for each $\alpha\in{\mathbb{N}}_0^n$ there exists 
$C_\alpha\in(0,\infty)$ such that 
$\big\|(\partial^\alpha K^L)(\cdot,t)\big\|_{[L^\infty({\mathbb{R}}^{n-1})]^{M\times M}}\leq C_\alpha t^{1-n-|\alpha|}$ 
for every $t>0$, together with an elementary observation to the effect that any bounded Lipschitz function in 
${\mathbb{R}}^{n-1}$ belongs to the H\"older space ${\mathcal{C}}^\theta({\mathbb{R}}^{n-1})$ 
for each exponent $\theta\in(0,1)$. In concert with the identification of the duals of Hardy spaces
(cf., e.g., \cite{GCRF85}) this shows that the pairings in \eqref{exist:u-123} are meaningful.

To prove \eqref{exist:Nu-Mf-Hp}, fix some $q\in(1,\infty)$ and assume first that the scalar components of $f$ 
are $(p,q)$-atoms. Hence, we need to consider 
\begin{equation}\label{eq:Fb}
u(x',t):=(P^L_t\ast a)(x'),\qquad\forall\,(x',t)\in{\mathbb{R}}^n_{+}, 
\end{equation}
where $a:\mathbb{R}^{n-1}\rightarrow\mathbb{C}^M$ is a $(p,q)$-atom (cf. \eqref{defi-atom-CM}).
Then, on account of \eqref{exist:Nu-Mf}, H\"older's inequality, the $L^{q}$-boundedness of the 
Hardy-Littlewood maximal operator, and the normalization of the atom we may write 
\begin{align}\label{eq:NBV1uj}
\int_{\sqrt{n}Q}\big({\mathcal{N}}_\kappa u\big)^p\,d{\mathcal{L}}^{n-1}
&\leq C\int_{\sqrt{n}Q}\big({\mathcal{M}}a\big)^p\,d{\mathcal{L}}^{n-1}
\nonumber\\[4pt]
&\leq C{\mathcal{L}}^{n-1}(Q)^{1-p/q}\Big(\int_{\sqrt{n}Q}\big({\mathcal{M}}a\big)^{q}\,d{\mathcal{L}}^{n-1}\Big)^{p/q}
\nonumber\\[4pt]
&\leq C{\mathcal{L}}^{n-1}(Q)^{1-p/q}\Big(\int_{\mathbb{R}^{n-1}}\big({\mathcal{M}}a\big)^q\,d{\mathcal{L}}^{n-1}\Big)^{p/q}
\nonumber\\[4pt]
&\leq C{\mathcal{L}}^{n-1}(Q)^{1-p/q}\|a\|_{[L^q(\mathbb{R}^{n-1})]^M}^p\leq C,
\end{align}
for some constant $C\in(0,\infty)$ depending only on $n,L,\kappa,p,q$.
To proceed, fix an arbitrary point $x'\in{\mathbb{R}}^{n-1}\setminus\sqrt{n}Q$.
If $\ell(Q)$ and $x'_Q$ are, respectively, the side-length and center of the cube $Q$, 
this choice entails 
\begin{equation}\label{eq:rEEb}
|z'-x_Q'|\leq\max\{\kappa,2\}\big(t+|z'-\xi'|\big),
\quad\forall\,(z',t)\in\Gamma_\kappa(x'),\,\,\forall\,\xi'\in Q.
\end{equation}
Indeed, if $(z',t)\in\Gamma_\kappa(x')$ and $\xi'\in Q$ then, first, 
$|z'-x'_Q|\leq |z'-\xi'|+|\xi'-x'_Q|$ and, second, 
$|\xi'-x'_Q|\leq\frac{\sqrt{n}}{2}\ell(Q)\leq\frac{1}{2}|x'-x'_Q|\leq\frac{1}{2}(|x'-z'|+|z'-x'_Q|)
\leq\frac{1}{2}(\kappa t+|z'-x'_Q|)$, from which \eqref{eq:rEEb} follows. 
Next, using \eqref{eq:Gvav7g5}, the vanishing moment condition for the atom, the Mean Value Theorem 
together with \eqref{eq:Kest} and \eqref{eq:rEEb}, H\"older's inequality and, finally, the 
support and normalization of the atom, for each $(z',t)\in\Gamma_\kappa(x')$ we may estimate 
\begin{align}\label{rrff4rf}
|(P^L_t\ast a)(z')| & =\Big|\int_{{\mathbb{R}}^{n-1}}\big[K^L(z'-y',t)-K^L(z'-x'_Q,t)\big]a(y')\,dy'\Big|
\nonumber\\[4pt]
&\leq\int_{Q}\big|K^L(z'-y',t)-K^L(z'-x'_Q,t)\big||a(y')|\,dy'
\nonumber\\[4pt]
&\leq C\frac{\ell(Q)}{\big(t+|z'-x'_Q|\big)^n}\int_{Q}|a(y')|\,dy'
\nonumber\\[4pt]
&\leq C\frac{\ell(Q)}{\big(t+|z'-x'_Q|\big)^n}\,
{\mathcal{L}}^{n-1}(Q)^{1-1/q}\|a\|_{[L^q(\mathbb{R}^{n-1})]^M}
\nonumber\\[4pt]
&\leq\frac{C{\mathcal{L}}^{n-1}(Q)^{1-1/p}\ell(Q)}{\big(t+|z'-x'_Q|\big)^n}.
\end{align}
In turn, \eqref{rrff4rf} implies that for each $x'\in{\mathbb{R}}^{n-1}\setminus\sqrt{n}Q$ we have
\begin{align}\label{rrff4rf.2}
\big({\mathcal{N}}_\kappa u\big)(x') &=\sup_{(z',t)\in\Gamma_\kappa(x')}|(P^L_t\ast a)(z')|
\\[6pt]
&\leq\sup_{(z',t)\in\Gamma_\kappa(x')}\frac{C{\mathcal{L}}^{n-1}(Q)^{1-1/p}\ell(Q)}{\big(t+|z'-x'_Q|\big)^n}
=\frac{C{\mathcal{L}}^{n-1}(Q)^{1-1/p}\ell(Q)}{|x'-x'_Q|^n},
\nonumber
\end{align}
hence
\begin{equation}\label{eq:NBVGaa}
\int_{{\mathbb{R}}^{n-1}\setminus\sqrt{n}Q}\big({\mathcal{N}}_\kappa u\big)^p\,d{\mathcal{L}}^{n-1}
\leq C\int_{{\mathbb{R}}^{n-1}\setminus\sqrt{n}Q}\frac{{\mathcal{L}}^{n-1}(Q)^{p-1}\ell(Q)^p}{|x'-x'_Q|^{np}}\,dx'=C,
\end{equation}
for some constant $C\in(0,\infty)$ depending only on $n,L,p,q$. From \eqref{eq:NBV1uj} and 
\eqref{eq:NBVGaa} we deduce that whenever $u$ is as in \eqref{eq:Fb} then, for some constant $C\in(0,\infty)$ 
independent of the atom, 
\begin{equation}\label{eq:NBhf}
\int_{{\mathbb{R}}^{n-1}}\big({\mathcal{N}}_\kappa u\big)^p\,d{\mathcal{L}}^{n-1}\leq C.
\end{equation}

Next, consider the general case when the function $u$ is defined as in \eqref{exist:u-123} for some 
arbitrary $f\in\big[H^p({\mathbb{R}}^{n-1})\big]^M$. Writing $f$ as in \eqref{exist:u-123.Wer.3} then 
permits us to express (in view of the specific manner in which the duality pairing in \eqref{exist:u-123} 
manifests itself), for each fixed $(x',t)\in{\mathbb{R}}^n_{+}$, 
\begin{align}\label{exist:u-123.aDS}
u(x',t)=(P^L_t\ast f)(x')=\sum_{j=1}^\infty\lambda_j(P^L_t\ast a_j)(x')=\sum_{j=1}^\infty\lambda_ju_j(x',t),
\end{align}
where $u_j(x',t):=(P^L_t\ast a_j)(x')$ for each $j\in{\mathbb{N}}$. Consequently, based on the 
sublinearity of the nontangential maximal operator, the fact that $p<1$, the estimate established 
in \eqref{eq:NBhf} (presently used with $a:=a_j$), and the quasi-optimality of the atomic decomposition 
for $f$, we may write 
\begin{align}\label{eq:NBhf.2iii}
\int_{{\mathbb{R}}^{n-1}}\big({\mathcal{N}}_\kappa u\big)^p\,d{\mathcal{L}}^{n-1}
&\leq\int_{{\mathbb{R}}^{n-1}}\Big(\sum_{j=1}^\infty{\mathcal{N}}_\kappa(\lambda_ju_j)\Big)^p\,d{\mathcal{L}}^{n-1}
\nonumber\\[6pt]
&\leq\int_{{\mathbb{R}}^{n-1}}\sum_{j=1}^\infty\big({\mathcal{N}}_\kappa(\lambda_ju_j)\big)^p\,d{\mathcal{L}}^{n-1}
\nonumber\\[6pt]
&=\sum_{j=1}^\infty|\lambda_j|^p\int_{{\mathbb{R}}^{n-1}}\big({\mathcal{N}}_\kappa u_j\big)^p\,d{\mathcal{L}}^{n-1}
\nonumber\\[6pt]
&\leq C\sum_{j=1}^\infty|\lambda_j|^p\leq C\|f\|^p_{[H^p({\mathbb{R}}^{n-1})]^M}.
\end{align}
This proves \eqref{exist:Nu-Mf-Hp}.

Finally, the claims in the first line of \eqref{exist:u2-Hp} are seen by differentiating inside the
duality bracket, while the existence of the nontangential boundary trace in the second line of \eqref{exist:u2-Hp} 
is a consequence of the corresponding result in \eqref{exist:u2}, the estimate in \eqref{exist:Nu-Mf-Hp}, the density 
of $H^p({\mathbb{R}}^{n-1})\cap L^2({\mathbb{R}}^{n-1})$ in $H^p({\mathbb{R}}^{n-1})$, and a well-known 
abstract principle in harmonic analysis (see, e.g., \cite[Theorem~2.2, p.\,27]{Duoan}, 
\cite[Theorem~3.12, p.\,60]{SW} for results of similar flavor).
\end{proof}

Let $L$ be an $M\times M$ system with constant complex coefficients as in \eqref{L-def}-\eqref{L-ell.X} and 
fix an integrability exponent $p\in(1,\infty)$. From items {\it (b)}-{\it (c)} in Theorem~\ref{thm:Poisson.II} 
we then see that the family $T=\{T(t)\}_{t\geq 0}$ where $T(0):=I$, the identity operator on 
$\big[L^p({\mathbb{R}}^{n-1})\big]^M$ and, for each $t>0$,
\begin{equation}\label{eq:Taghb8}
\begin{array}{c}
T(t):\big[L^p({\mathbb{R}}^{n-1})\big]^M\longrightarrow\big[L^p({\mathbb{R}}^{n-1})\big]^M,
\\[6pt]
\big(T(t)f\big)(x'):=(P^L_t\ast f)(x')\,\text{ for all }\,
f\in\big[L^p({\mathbb{R}}^{n-1})\big]^M,\,\,x'\in{\mathbb{R}}^{n-1},
\end{array}
\end{equation}
is a $C_0$-semigroup on $\big[L^p({\mathbb{R}}^{n-1})\big]^M$, which satisfies
\begin{equation}\label{eq:Taghb8.77}
\sup_{t\geq 0}\big\|T(t)\big\|_{[L^p({\mathbb{R}}^{n-1})]^M\to[L^p({\mathbb{R}}^{n-1})]^M}<\infty.
\end{equation}

We now proceed to present several Fatou-type theorems and Poisson integral representation 
formulas for null-solutions of homogeneous constant complex coefficient elliptic second-order 
systems defined in ${\mathbb{R}}^n_{+}$ and subject to a variety of size conditions. 

\begin{theorem}\label{thm:FP.111}
Let $L$ be an $M\times M$ system with constant complex coefficients as in \eqref{L-def}-\eqref{L-ell.X}, 
and fix some aperture parameter $\kappa>0$. Suppose $u\in\big[\mathcal{C}^\infty(\mathbb{R}^n_{+})\big]^M$ 
satisfies $Lu=0$ in $\mathbb{R}_{+}^n$, as well as
\begin{equation}\label{UGav-5hH9i}
\int_{{\mathbb{R}}^{n-1}}\big({\mathcal{N}}^{(\varepsilon)}_\kappa u\big)(x')\frac{dx'}{1+|x'|^{n-1}}<\infty
\text{ for each fixed }\,\,\varepsilon>0,
\end{equation}
and also assume that there exists $\varepsilon_0>0$ such that the following finiteness integral condition holds:
\begin{equation}\label{u-integ-TR}
\int_{\mathbb{R}^{n-1}}\frac{\sup_{0<t<\varepsilon_0}|u(x',t)|}{1+|x'|^n}\,dx'<\infty.
\end{equation}

Then
\begin{equation}\label{Tafva.2222}
\left\{
\begin{array}{l}
\big(u\big|^{{}^{\kappa-{\rm n.t.}}}_{\partial{\mathbb{R}}^n_{+}}\big)(x')
\,\,\text{ exists at ${\mathcal{L}}^{n-1}$-a.e. point }\,\,x'\in{\mathbb{R}}^{n-1},
\\[10pt]
\displaystyle
u\big|^{{}^{\kappa-{\rm n.t.}}}_{\partial{\mathbb{R}}^n_{+}}\,\,\text{ belongs to the space }\,\,
\Big[L^1\Big({\mathbb{R}}^{n-1}\,,\,\frac{dx'}{1+|x'|^{n}}\Big)\Big]^M,
\\[12pt]
u(x',t)=\Big(P^L_t\ast\big(u\big|^{{}^{\kappa-{\rm n.t.}}}_{\partial{\mathbb{R}}^n_{+}}\big)\Big)(x')
\,\,\text{ for each }\,\,(x',t)\in{\mathbb{R}}^n_{+},
\end{array}
\right.
\end{equation}
where $P^L$ is the Agmon-Douglis-Nirenberg Poisson kernel in ${\mathbb{R}}^n_{+}$ associated with the system $L$  
as in Theorem~\ref{thm:Poisson}. In particular, from \eqref{nkc-EE-4}, \eqref{Tafva.2222}, and \eqref{exist:Nu-Mf} 
it follows that there exists a constant $C=C(L,\kappa)\in(0,\infty)$ with the property that
\begin{equation}\label{Tafva.2222.iii.1233}
\big|u\big|^{{}^{\kappa-{\rm n.t.}}}_{\partial{\mathbb{R}}^n_{+}}\big|
\leq{\mathcal{N}}_\kappa u\leq C{\mathcal{M}}\big(u\big|^{{}^{\kappa-{\rm n.t.}}}_{\partial{\mathbb{R}}^n_{+}}\big)
\,\,\text{ in }\,\,{\mathbb{R}}^{n-1}.
\end{equation}
\end{theorem}

It is natural to think of \eqref{Tafva.2222.iii.1233}, which implies that for almost 
every point $x'\in{\mathbb{R}}^{n-1}$ the supremum of the function $u$ in the cone 
$\Gamma_{\kappa}(x')$ lies in between the (absolute value of the) boundary trace 
$u\big|^{{}^{\kappa-{\rm n.t.}}}_{\partial{\mathbb{R}}^n_{+}}$ evaluated at $x'$
and a fixed multiple of the Hardy-Littlewood operator acting 
on the boundary trace $u\big|^{{}^{\kappa-{\rm n.t.}}}_{\partial{\mathbb{R}}^n_{+}}$ at $x'$, 
as some type of ``Pointwise Maximum Principle.'' 

Theorem~\ref{thm:FP.111} is optimal from multiple perspectives. First, observe from 
item $(a)$ of Theorem~\ref{thm:Poisson.II} and item $(c)$ of Theorem~\ref{thm:Poisson}
that for each $a\in{\mathbb{C}}^M\setminus\{0\}$ the function 
\begin{equation}\label{FCT-TR}
u_a(x',t):=t^{1-n}P^L(x'/t)a=K^L(x',t)a,\,\,\text{ for each }\,\,(x',t)\in\mathbb{R}^n_{+},
\end{equation}
satisfies $u_a\in\big[\mathcal{C}^{\infty}(\overline{\mathbb{R}^n_{+}}\setminus\{0\})\big]^M$, 
$Lu_a=0$ in $\mathbb{R}_{+}^n$, and $\Big(u_a\big|^{{}^{\kappa-{\rm n.t.}}}_{\partial{\mathbb{R}}^{n}_{+}}\Big)(x')=0$ 
for every aperture parameter $\kappa>0$ and every point $x'\in{\mathbb{R}}^{n-1}\setminus\{0'\}$. 
Moreover, $u_a$ is not identically zero since $\int_{\mathbb{R}^{n-1}}u_a(x',t)\,dx'=a$ for each 
$t>0$ by \eqref{eq:IG6gy.2}. As such, the Poisson integral representation formula in the 
last line of \eqref{Tafva.2222} fails for each $a\in{\mathbb{C}}^M\setminus\{0\}$. 
Let us also observe that having 
\begin{equation}\label{ijfdghba}
|u_a(y',t)|\leq C|a|t(t^2+|y'|^2)^{-n/2}\,\,\text{ for each }\,\,
(y',t)\in{\mathbb{R}}^n_{+}
\end{equation}
entails that for each $\varepsilon>0$ there exists $C_{a,\varepsilon}\in(0,\infty)$ such that 
\begin{equation}\label{UGav-5hH9i-agg}
\big({\mathcal{N}}^{(\varepsilon)}_\kappa u_a\big)(x')\leq\frac{C_{a,\varepsilon}}{1+|x'|^{n-1}}
\,\,\text{ for each }\,\,x'\in{\mathbb{R}}^{n-1}.
\end{equation}
Hence, condition \eqref{UGav-5hH9i} is presently satisfied by each $u_a$. 
In light of Theorem~\ref{thm:FP.111} the finiteness integral condition stipulated 
in \eqref{u-integ-TR} then necessarily should fail for each $u_a$. To check directly that 
this is the case, recall from \eqref{eq:IG6gy.2-Lambda} that for each 
$a\in{\mathbb{C}}^M\setminus\{0\}$ there exists some $\lambda>0$ such that 
$\int_{S^{n-2}}\big|P^L(\lambda\omega)a\big|\,d\omega\in(0,\infty)$. In turn, 
this permits us to estimate
\begin{align}\label{y65trta}
\int_{\mathbb{R}^{n-1}}&\frac{\sup_{0<t<\varepsilon_0}|u_a(x',t)|}{1+|x'|^n}\,dx'
\geq\int_{B_{n-1}(0',\lambda\varepsilon_0)}\frac{|u_a(x',|x'|/\lambda)|}{1+|x'|^n}\,dx'
\nonumber\\[6pt]
&=\int_{B_{n-1}(0',\lambda\varepsilon_0)}\frac{(|x'|/\lambda)^{1-n}\big|P^L(\lambda x'/|x'|)a\big|}{1+|x'|^n}\,dx'
\\[6pt]
&=\Big(\int_{S^{n-2}}\big|P^L(\lambda\omega)a\big|\,d\omega\Big)
\Big(\int_0^{\lambda\varepsilon_0}\frac{\lambda^{n-1}}{\rho(1+\rho^n)}\,d\rho\Big)=\infty,
\nonumber
\end{align}
using \eqref{FCT-TR} and passing to polar coordinates. Thus, \eqref{u-integ-TR} fails for each $u_a$. 

Second, the absolute integrability condition \eqref{u-integ-TR} may not be in general replaced by membership to the 
corresponding weak Lebesgue space. For example, in the case $L:=\Delta$, the Laplacian in ${\mathbb{R}}^n$, 
if \eqref{u-integ-TR} is weakened to the demand that 
\begin{equation}\label{u-integ-TR.adfg.jk}
\begin{array}{c}
\mathbb{R}^{n-1}\ni x'\mapsto\frac{\sup_{0<t<\varepsilon_0}|u(x',t)|}{1+|x'|^n}\in[0,\infty]
\\[6pt]
\text{is a function belonging to }\,\,L^{1,\infty}({\mathbb{R}}^{n-1})
\end{array}
\end{equation}
then Theorem~\ref{thm:FP.111} may fail. Indeed, this may be seen by considering the nonzero harmonic function 
$u(x',t)=t(t^2+|x'|^2)^{-n/2}$ for each $(x',t)\in{\mathbb{R}}^n_{+}$, which satisfies \eqref{UGav-5hH9i} and 
\eqref{u-integ-TR.adfg.jk}. However, the Poisson integral representation formula in the last line of 
\eqref{Tafva.2222} fails since $\Big(u\big|^{{}^{\kappa-{\rm n.t.}}}_{\partial{\mathbb{R}}^{n}_{+}}\Big)(x')=0$ 
for every $x'\in{\mathbb{R}}^{n-1}\setminus\{0'\}$.

Third, one cannot relax the formulation of the finiteness integral condition
\eqref{u-integ-TR} by placing the supremum outside the integral sign. To see this, 
fix an arbitrary $a\in{\mathbb{C}}^M\setminus\{0\}$ and take $u_a$ as in \eqref{FCT-TR}. 
Then, thanks to \eqref{eq:IG6gy}, we have 
\begin{align}\label{hgggf}
\sup_{0<t<\varepsilon_0}\int_{\mathbb{R}^{n-1}}\frac{|u_a(x',t)|}{1+|x'|^n}\,dx'
&\leq\sup_{0<t<\varepsilon_0}\int_{\mathbb{R}^{n-1}}|u_a(x',t)|\,dx'
\nonumber\\[6pt]
&=\sup_{0<t<\varepsilon_0}\int_{\mathbb{R}^{n-1}}\big|P_t^{L}(x')a\big|\,dx'
\nonumber\\[6pt]
&\leq|a|\int_{\mathbb{R}^{n-1}}\big|P^{L}(x')\big|\,dx'<\infty.
\end{align}
Yet, again, the Poisson representation formula in the last line of \eqref{Tafva.2222} fails.

\vskip 0.06in
One notable consequence of Theorem~\ref{thm:FP.111} is the Fatou-type theorem and 
its associated Poisson integral formula presented below.

\begin{theorem}\label{thm:FP}
Let $L$ be an $M\times M$ system with constant complex coefficients as in 
\eqref{L-def}-\eqref{L-ell.X}, and fix some aperture parameter $\kappa>0$. Then having
\begin{equation}\label{jk-lm-jhR-LLL-HM-RN.w}
\left\{
\begin{array}{l}
u\in\big[{\mathcal{C}}^{\infty}({\mathbb{R}}^n_{+})\big]^M,\quad Lu=0\,\,\text{ in }\,\,{\mathbb{R}}^n_{+},
\\[8pt]
\displaystyle
\int_{\mathbb{R}^{n-1}}\big({\mathcal{N}}_{\kappa}u\big)(x')\,\frac{dx'}{1+|x'|^{n-1}}<\infty,
\end{array}
\right.
\end{equation}
implies that 
\begin{equation}\label{Tafva.2222.iii}
\left\{
\begin{array}{l}
\big(u\big|^{{}^{\kappa-{\rm n.t.}}}_{\partial{\mathbb{R}}^n_{+}}\big)(x')
\,\,\text{ exists at ${\mathcal{L}}^{n-1}$-a.e. point }\,\,x'\in{\mathbb{R}}^{n-1},
\\[10pt]
\displaystyle
u\big|^{{}^{\kappa-{\rm n.t.}}}_{\partial{\mathbb{R}}^n_{+}}\,\,\text{ belongs to the space }\,\,
\Big[L^1\Big({\mathbb{R}}^{n-1}\,,\,\frac{dx'}{1+|x'|^{n-1}}\Big)\Big]^M,
\\[12pt]
u(x',t)=\Big(P^L_t\ast\big(u\big|^{{}^{\kappa-{\rm n.t.}}}_{\partial{\mathbb{R}}^n_{+}}\big)\Big)(x')
\,\,\text{ for each }\,\,(x',t)\in{\mathbb{R}}^n_{+},
\end{array}
\right.
\end{equation}
where $P^L$ is the Agmon-Douglis-Nirenberg Poisson kernel in ${\mathbb{R}}^n_{+}$ associated with the system $L$  
as in Theorem~\ref{thm:Poisson}. In particular, there exists a constant $C=C(L,\kappa)\in(0,\infty)$ 
such that the following Pointwise Maximum Principle holds:
\begin{equation}\label{Taf-UHN}
\big|u\big|^{{}^{\kappa-{\rm n.t.}}}_{\partial{\mathbb{R}}^n_{+}}\big|
\leq{\mathcal{N}}_\kappa u\leq C{\mathcal{M}}\big(u\big|^{{}^{\kappa-{\rm n.t.}}}_{\partial{\mathbb{R}}^n_{+}}\big)
\,\,\text{ in }\,\,{\mathbb{R}}^{n-1}.
\end{equation}
\end{theorem}

\vskip 0.08in
\begin{proof}[Proof of the fact that Theorem~\ref{thm:FP.111} implies Theorem~\ref{thm:FP}]
Given $u$ as in \eqref{jk-lm-jhR-LLL-HM-RN.w}, we have 
\begin{equation}\label{u-integ-TR.atr}
\int_{\mathbb{R}^{n-1}}\frac{\sup_{0<t<\varepsilon_0}|u(x',t)|}{1+|x'|^n}\,dx'
\leq C_n\int_{\mathbb{R}^{n-1}}\big({\mathcal{N}}_{\kappa}u\big)(x')\,\frac{dx'}{1+|x'|^{n-1}}<\infty.
\end{equation}
In view of this and \eqref{NT-Fct-EP.anB}, we conclude that the conditions stipulated in 
\eqref{UGav-5hH9i}-\eqref{u-integ-TR} are valid. As such, Theorem~\ref{thm:FP.111} guarantees 
that the properties listed in the first and third lines of \eqref{Tafva.2222.iii} hold. 
In addition, thanks to \eqref{nkc-EE-4} we now have the membership claimed in the middle 
line of \eqref{Tafva.2222.iii}. 
\end{proof}

A direct, self-contained proof of Theorem~\ref{thm:FP} (without having to rely on Theorem~\ref{thm:FP.111}) 
has been given in \cite{Madrid}. Here we shall indicate how Theorem~\ref{thm:FP} self-improves to 
Theorem~\ref{thm:FP.111}. In the process, we shall need the following weak-* convergence result from \cite{SCGC}. 

\begin{lemma}\label{WLLL}
Suppose $\{f_j\}_{j\in{\mathbb{N}}}\subseteq L^1\big({\mathbb{R}}^{n-1}\,,\,\frac{dx'}{1+|x'|^{n}}\big)$ 
is a sequence of functions satisfying 
\begin{equation}\label{523563fve-NNN}
\int_{\mathbb{R}^{n-1}}\frac{\sup_{j\in{\mathbb{N}}}|f_j(x')|}{1+|x'|^n}\,dx'<\infty.
\end{equation}

Then there exist $f\in L^1\big(\mathbb{R}^{n-1},\frac{dx'}{1+|x'|^n}\big)$ and 
a sub-sequence $\big\{f_{j_k}\big\}_{k\in{\mathbb{N}}}$ of $\{f_j\}_{j\in{\mathbb{N}}}$ with the property that
\begin{equation}\label{2135racvs-NNN}
\lim_{k\to\infty}\int_{\mathbb{R}^{n-1}}\phi(y')\,f_{j_k}(y')\,\frac{dy'}{1+|y'|^n}
=\int_{\mathbb{R}^{n-1}}\phi(y')\,f(y')\,\frac{dy'}{1+|y'|^n}
\end{equation}
for every function $\phi$ belonging to $\mathcal{C}^0_b(\mathbb{R}^{n-1})$, the space of ${\mathbb{C}}$-valued 
continuous and bounded functions in $\mathbb{R}^{n-1}$. 
\end{lemma}

We are ready to provide a proof of Theorem~\ref{thm:FP.111} which relies on Theorem~\ref{thm:FP}.

\vskip 0.08in
\begin{proof}[Proof of Theorem~\ref{thm:FP.111}]
For each $\varepsilon>0$ define $u_\varepsilon(x',t):=u(x',t+\varepsilon)$ for each $(x',t)\in\overline{{\mathbb{R}}^n_{+}}$.
Also, set $f_\varepsilon(x'):=u(x',\varepsilon)$ for each $x'\in{\mathbb{R}}^{n-1}$.  
Since we have ${\mathcal{N}}_{\kappa}u_\varepsilon\leq{\mathcal{N}}^{(\varepsilon)}_{\kappa}u$ 
on $\mathbb{R}^{n-1}$, we conclude that 
\begin{equation}\label{jk-lm-jhR-LLL-HM-RN.w.123}
\begin{array}{c}
u_\varepsilon\in\big[{\mathcal{C}}^{\infty}(\overline{{\mathbb{R}}^n_{+}})\big]^M,\quad 
u_\varepsilon\big|^{{}^{\kappa-{\rm n.t.}}}_{\partial{\mathbb{R}}^n_{+}}=f_\varepsilon\,\,\text{ on }\,\,\mathbb{R}^{n-1},
\quad Lu_\varepsilon=0\,\,\text{ in }\,\,{\mathbb{R}}^n_{+},
\\[8pt]
\displaystyle
\int_{\mathbb{R}^{n-1}}\big({\mathcal{N}}_{\kappa}u_\varepsilon\big)(x')\,\frac{dx'}{1+|x'|^{n-1}}
\leq\int_{\mathbb{R}^{n-1}}\big({\mathcal{N}}^{(\varepsilon)}_{\kappa}u\big)(x')\,\frac{dx'}{1+|x'|^{n-1}}<\infty.
\end{array}
\end{equation}
As such, Theorem~\ref{thm:FP} applies to each $u_\varepsilon$, and the Poisson integral representation formula in 
the last line of \eqref{Tafva.2222.iii} presently guarantees that for each $\varepsilon>0$ we have
\begin{align}\label{TaajYGa}
u(x',t+\varepsilon) &=u_\varepsilon(x',t)=\big(P^L_t\ast f_\varepsilon\big)(x')
\nonumber\\[6pt]
&=\int_{{\mathbb{R}}^{n-1}}P^L_t(x'-y')f_\varepsilon(y')\,dy'
\,\,\text{ for each }\,\,(x',t)\in{\mathbb{R}}^n_{+}.
\end{align}
On the other hand, property \eqref{u-integ-TR} entails  
\begin{equation}\label{eq:16t44.iii}
\int_{\mathbb{R}^{n-1}}\frac{\sup_{0<\varepsilon<\varepsilon_0}|f_\varepsilon(x')|}{1+|x'|^n}\,dx'
=\int_{\mathbb{R}^{n-1}}\frac{\sup_{0<\varepsilon<\varepsilon_0}|u(x',\varepsilon)|}{1+|x'|^n}\,dx'<\infty.
\end{equation}
Granted this finiteness property, the weak-$\ast$ convergence result recalled in Lemma~\ref{WLLL} 
may be used for the sequence $\big\{f_{\varepsilon_0/(2j)}\big\}_{j\in{\mathbb{N}}}\subset
\big[L^1\big({\mathbb{R}}^{n-1}\,,\,\tfrac{dx'}{1+|x'|^n}\big)\big]^M$ to conclude that there 
exist some function $f\in\big[L^1\big({\mathbb{R}}^{n-1}\,,\,\tfrac{dx'}{1+|x'|^n}\big)\big]^M$ 
and some sequence $\{\varepsilon_k\}_{k\in{\mathbb{N}}}\subset(0,\varepsilon_0)$ which converges 
to zero, such that 
\begin{equation}\label{eq:16t44.MAD}
\lim_{k\to\infty}\int_{{\mathbb{R}}^{n-1}}\phi(y')f_{\varepsilon_k}(y')\frac{dy'}{1+|y'|^{n}}
=\int_{{\mathbb{R}}^{n-1}}\phi(y')f(y')\frac{dy'}{1+|y'|^{n}}
\end{equation}
for every continuous bounded function $\phi:{\mathbb{R}}^{n-1}\to{\mathbb{C}}^{M\times M}$.
Estimate \eqref{eq:IG6gy} ensures for each fixed point $(x',t)\in{\mathbb{R}}^n_{+}$ the assignment 
\begin{equation}\label{Fvabbb-7tF.Tda}
\begin{array}{c}
\displaystyle
{\mathbb{R}}^{n-1}\ni y'\mapsto\phi(y'):=(1+|y'|^{n})P^L_t(x'-y')\in{\mathbb{C}}^{M\times M}
\\[6pt]
\text{is a bounded, continuous, ${\mathbb{C}}^{M\times M}$-valued function}.
\end{array}
\end{equation}
At this stage, from \eqref{TaajYGa} and \eqref{eq:16t44.MAD} 
used for the function $\phi$ defined in \eqref{Fvabbb-7tF.Tda} we obtain 
(bearing in mind that $u$ is continuous in ${\mathbb{R}}^n_{+}$) that 
\begin{align}\label{AA-lm-jLL-HM-RN.ppp.wsa.2}
u(x',t)=\int_{{\mathbb{R}}^{n-1}}P^L_t(x'-y')f(y')\,dy'\,\,\,\text{ for each }\,\,x=(x',t)\in{\mathbb{R}}^n_{+}.
\end{align}
With this in hand, and recalling that 
$f\in\Big[L^1\big({\mathbb{R}}^{n-1}\,,\,\tfrac{dx'}{1+|x'|^{n}}\big)\Big]^M$, 
we may invoke Theorem~\ref{thm:Poisson.II} (cf. \eqref{exist:u2}) to conclude that
\begin{equation}\label{exist:u2.b}
\text{$u\big|^{{}^{\kappa-{\rm n.t.}}}_{\partial\mathbb{R}^{n}_{+}}$ exists and equals $f$ 
at ${\mathcal{L}}^{n-1}$-a.e. point in $\mathbb{R}^{n-1}$}.
\end{equation}
With this in hand, all conclusions in \eqref{Tafva.2222} are implied by 
\eqref{AA-lm-jLL-HM-RN.ppp.wsa.2}-\eqref{exist:u2.b}. 
\end{proof}

Moving on, we consider two families of semi-norms on the class of continuous functions 
in ${\mathbb{R}}^n_{+}$, namely 
\begin{equation}\label{semi-norms}
\|u\|_{*,\rho}:=\rho^{-1}\cdot\sup_{B(0,\rho)\cap\mathbb{R}^n_{+}}|u|\in[0,+\infty]
\,\,\text{ for each }\,\,\rho\in(0,\infty), 
\end{equation}
and
\begin{equation}\label{semi-norms-eps}
\|u\|_{*,\varepsilon,\rho}:=\rho^{-1}\cdot\sup_{\substack{\varepsilon<t<\rho\\ |x'|<\rho}}|u(x',t)|
\,\,\,\text{ whenever }\,\,0<\varepsilon<\rho<\infty.
\end{equation}
Whenever $\liminf\limits_{\rho\to\infty}\|u\|_{*,\rho}=0$ we shall say that $u$ has subcritical growth. 
In this vein, it is worth observing that
\begin{equation}\label{subcritical-EQUI}
\parbox{9.50cm}{a continuous function $u:\mathbb{R}^n_{+}\to\mathbb{C}$ 
has $\lim\limits_{\rho\to\infty}\|u\|_{*,\rho}=0$
if and only if $u$ is bounded on any bounded subset of ${\mathbb{R}}^n_{+}$ and $u(x)=o(|x|)$ 
as $x\in{\mathbb{R}}^n_{+}$ satisfies $|x|\to\infty$.} 
\end{equation}

In relation to the family of seminorms introduced in \eqref{semi-norms}-\eqref{semi-norms-eps}, 
let us also observe that for each continuous function $u:{\mathbb{R}}^n_{+}\to{\mathbb{C}}$ we have
\begin{equation}\label{fvqafr}
\|u\|_{*,\rho}=\frac1{\log 2}\int_\rho^{2\rho}\|u\|_{*,\rho}\,\frac{dt}{t}
\leq\frac2{\log 2}\int_\rho^{2\rho}\|u\|_{*,t}\,\frac{dt}{t}
\end{equation}
for each $\rho\in(0,\infty)$, and that 
\begin{equation}\label{q34t3g3a}
\|u\|_{*,\varepsilon,\rho}\leq\sqrt{2}\|u\|_{*,\sqrt{2}\rho}
\leq\rho^{-1}\|u\|_{L^\infty(\mathbb{R}^n_{+})}
\,\,\,\text{ whenever }\,\,0<\varepsilon<\rho<\infty.
\end{equation}

Our next major theorem is a novel Fatou-type result (plus a naturally accompanying Poisson integral 
representation formula), recently established in \cite{SCGC}, of the sort discussed below. 

\begin{theorem}\label{thm:Fatou}
Let $L$ be an $M\times M$ homogeneous constant complex coefficient elliptic second-order system in ${\mathbb{R}}^n$. 
Let $u\in\big[\mathcal{C}^\infty(\mathbb{R}^n_{+})\big]^M$ be such that $Lu=0$ in $\mathbb{R}_{+}^n$.
In addition, assume  
\begin{equation}\label{subcritical:mild}
\liminf_{\rho\to\infty}\|u\|_{*,\varepsilon,\rho}=0\,\,\text{ for each fixed }\,\,\varepsilon>0,
\end{equation}
and suppose that there exists $\varepsilon_0>0$ such that the following finiteness integral condition holds:
\begin{equation}\label{u-integ}
\int_{\mathbb{R}^{n-1}}\frac{\sup_{0<t<\varepsilon_0}|u(x',t)|}{1+|x'|^n}\,dx'<\infty.
\end{equation}

Then, for each aperture parameter $\kappa>0$,
\begin{align}\label{Tafva.2222-ijk}
\begin{array}{l}
\big(u\big|^{{}^{\kappa-{\rm n.t.}}}_{\partial{\mathbb{R}}^n_{+}}\big)(x')
\,\,\text{ exists at ${\mathcal{L}}^{n-1}$-a.e. point }\,\,x'\in{\mathbb{R}}^{n-1},
\\[10pt]
\displaystyle
u\big|^{{}^{\kappa-{\rm n.t.}}}_{\partial{\mathbb{R}}^n_{+}}\,\,\text{ belongs to the space }\,\,
\Big[L^1\Big({\mathbb{R}}^{n-1}\,,\,\frac{dx'}{1+|x'|^{n}}\Big)\Big]^M,
\\[12pt]
u(x',t)=\Big(P^L_t\ast\big(u\big|^{{}^{\kappa-{\rm n.t.}}}_{\partial{\mathbb{R}}^n_{+}}\big)\Big)(x')
\,\,\text{ for each }\,\,(x',t)\in{\mathbb{R}}^n_{+}.
\end{array}
\end{align}
As a consequence, there exists a constant $C=C(L,\kappa)\in(0,\infty)$ with the property that 
the following Pointwise Maximum Principle holds:
\begin{equation}\label{Taf-UHN.ER.2}
\big|u\big|^{{}^{\kappa-{\rm n.t.}}}_{\partial{\mathbb{R}}^n_{+}}\big|
\leq{\mathcal{N}}_\kappa u\leq C{\mathcal{M}}\big(u\big|^{{}^{\kappa-{\rm n.t.}}}_{\partial{\mathbb{R}}^n_{+}}\big)
\,\,\text{ in }\,\,{\mathbb{R}}^{n-1}.
\end{equation}
\end{theorem}

The Fatou-type result established in Theorem~\ref{thm:Fatou} is optimal from a multitude 
of perspectives. First, the mildly weaker version of the subcritical growth condition stated 
in \eqref{subcritical:mild} cannot be relaxed. 
Indeed, fix $a\in{\mathbb{C}}^M\setminus\{0\}$ and consider the function $u(x',t):=ta$
for each $(x',t)\in\partial{\mathbb{R}}^n_{+}$. Then $\|u\|_{*,\varepsilon,\rho}=|a|>0$, 
hence \eqref{subcritical:mild} fails while \eqref{u-integ} and the first two properties listed in 
\eqref{Tafva.2222-ijk} hold. Nonetheless, the Poisson integral representation formula claimed in 
the last line of \eqref{Tafva.2222-ijk} fails 
(since $u\big|^{{}^{\kappa-{\rm n.t.}}}_{\partial{\mathbb{R}}^n_{+}}=0$ 
everywhere on ${\mathbb{R}}^{n-1}$ whereas $u$ is nonzero). 

Second, the finiteness integral condition \eqref{u-integ} may not be dropped. 
To justify this claim, bring in the Poisson kernel $P^L:\mathbb{R}^{n-1}\to\mathbb{C}^{M\times M}$ 
associated with the system $L$ as in Theorem~\ref{thm:Poisson} and, having fixed some 
$a\in{\mathbb{C}}^M\setminus\{0\}$, consider the function  $u_a$ defined as in \eqref{FCT-TR}. 
In addition to the properties this function enjoys mentioned earlier, for each $\varepsilon>0$ 
fixed we have $\|u_a\|_{*,\varepsilon,\rho}\leq C|a|\rho^{-1}\varepsilon^{1-n}\to 0$ as $\rho\to\infty$. 
However, the Poisson integral representation formula claimed in the last line of 
\eqref{Tafva.2222-ijk} 
obviously fails. The source of the failure is the fact that \eqref{u-integ} does not presently materialize 
(as already noted in \eqref{y65trta}). 

Third, as seen from \eqref{hgggf}, one cannot relax the formulation of the finiteness integral 
condition \eqref{u-integ} by placing the supremum outside the integral sign. 

\medskip 

In particular, Theorem~\ref{thm:Fatou} implies a uniqueness result, to the effect that whenever 
$L$ is an $M\times M$ homogeneous constant complex coefficient elliptic second-order system in ${\mathbb{R}}^n$
one has
\begin{equation}\label{Fatou-Uniqueness}
\left.
\begin{array}{r}
u\in\big[{\mathcal{C}}^\infty(\mathbb{R}^{n}_{+})\big]^M,\,\,Lu=0\,\,\mbox{ in }\,\,\mathbb{R}^{n}_{+}
\\[4pt]
\text{$u$ satisfies both \eqref{subcritical:mild} and \eqref{u-integ}}
\\[6pt]
u\big|^{{}^{\kappa-{\rm n.t.}}}_{\partial{\mathbb{R}}^{n}_{+}}=0
\,\,\text{ at ${\mathcal{L}}^{n-1}$-a.e. point in  }\,\,\mathbb{R}^{n-1}
\end{array}
\right\}
\Longrightarrow u\equiv 0\,\,\text{ in }\,\,{\mathbb{R}}^n_{+}.
\end{equation}
This should be compared with the following uniqueness result within the class 
of null-solutions of the system $L$ exhibiting subcritical growth (also established in \cite{SCGC}). 

\begin{theorem}\label{thm:uniq-subcritical}
Let $L$ be an $M\times M$ homogeneous constant complex coefficient elliptic second-order system in ${\mathbb{R}}^n$
and fix an aperture parameter $\kappa>0$. Assume $u\in\big[\mathcal{C}^\infty(\mathbb{R}^n_{+})\big]^M$ is such that 
$Lu=0$ in $\mathbb{R}_{+}^n$, $u$ satisfies the subcritical growth condition 
\begin{equation}\label{subcritical}
\liminf_{\rho\to\infty}\|u\|_{*,\rho}=0,
\end{equation}
and that $u\big|^{{}^{\kappa-{\rm n.t.}}}_{\partial{\mathbb{R}}^{n}_{+}}=0$ 
at ${\mathcal{L}}^{n-1}$-a.e. point in $\mathbb{R}^{n-1}$. Then $u\equiv 0$ in $\mathbb{R}^n_{+}$.
\end{theorem}

We wish to note that the subcritical growth condition \eqref{subcritical} is sharp. Concretely, 
while $\liminf\limits_{\rho\to\infty}\|u\|_{*,\rho}$ always exists and is a non-negative 
number, its failure to vanish does not force $u$ to be identically zero. Indeed, for any 
$a\in{\mathbb{C}}^M\setminus\{0\}$ the function $u(x',t):=ta$ satisfies 
$u\in\big[\mathcal{C}^\infty(\overline{\mathbb{R}^n_{+}})\big]^M$, $Lu=0$ in $\mathbb{R}_{+}^n$, 
and $u\big|^{{}^{\kappa-{\rm n.t.}}}_{\partial{\mathbb{R}}^{n}_{+}}=0$ everywhere in $\mathbb{R}^{n-1}$ 
(for any aperture parameter $\kappa>0$). This being said, $\sup_{B(0,\rho)\cap\mathbb{R}^n_{+}}|u|=|a|\rho$ 
for each $\rho>0$, hence
\begin{equation}\label{subcritical-counter}
\liminf_{\rho\to\infty}\Big(\rho^{-1}\sup_{B(0,\rho)\cap\mathbb{R}^n_{+}}|u|\Big)=|a|>0.
\end{equation}

A comment on the genesis of the subcritical growth condition \eqref{subcritical} is in also order. 
Suppose $L:=\Delta$ (the Laplacian in ${\mathbb{R}}^n$) and one is interested in establishing 
a uniqueness result in the class of functions
\begin{equation}\label{UUU.uuu}
u\in{\mathcal{C}}^\infty({\mathbb{R}}^n_{+})\cap{\mathcal{C}}^0(\overline{{\mathbb{R}}^n_{+}})
\,\,\text{ with }\,\,\Delta u=0\,\,\text{ in }\,\,{\mathbb{R}}^n_{+},
\end{equation}
to the effect that the boundary trace $u\big|_{\partial{\mathbb{R}}^n_{+}}$ determines $u$. 
Since $u(x',t)=t$ for each $(x',t)\in{\mathbb{R}}^n_{+}$ is a counterexample, a further demand must 
be imposed, in addition to \eqref{UUU.uuu}, to rule out this pathological example. To identify this demand, 
consider a function $u$ as in \eqref{UUU.uuu} which satisfies $u\big|_{\partial{\mathbb{R}}^n_{+}}=0$. 
Then Schwarz's reflection principle ensures that
\begin{equation}\label{UUU.uuu.2}
\widetilde{u}(x',t):=\left\{
\begin{array}{ll}
u(x',t) &\text{ if }\,\,t\geq 0,
\\[4pt]
-u(x',-t) &\text{ if }\,\,t>0,
\end{array}
\right.
\qquad\forall\,(x',t)\in{\mathbb{R}}^n,
\end{equation}
is a harmonic function in ${\mathbb{R}}^n$. Interior estimates then imply the existence of a
dimensional constant $C_n\in(0,\infty)$ with the property that 
\begin{equation}\label{UUU.uuu.3}
\big|(\nabla\widetilde{u})(x)\big|\leq C_n\rho^{-1}\cdot\sup_{B(x,\rho)}\big|\widetilde{u}\big|
\,\,\text{ for each }\,\,x\in{\mathbb{R}}^n\,\,\text{ and }\,\,\rho>0.
\end{equation}
In this context, it is clear that the subcritical growth condition \eqref{subcritical} is a 
quantitatively optimal property guaranteeing the convergence to zero of the right-hand side
of the inequality in \eqref{UUU.uuu.3} as $\rho\to\infty$, for each $x\in{\mathbb{R}}^n$ fixed.
And this is precisely what is needed here since this further implies $\nabla\widetilde{u}\equiv 0$ 
in ${\mathbb{R}}^n$, which ultimately forces $u\equiv 0$ in ${\mathbb{R}}^n$.

The new challenges in Theorem~\ref{thm:uniq-subcritical} stem from the absence of a Schwarz's reflection principle
in the more general class of systems we are currently considering, and the lack of continuity of the function 
at boundary points. Our proof of Theorem~\ref{thm:uniq-subcritical} circumvents these obstacles by making 
use of Agmon-Douglis-Nirenberg estimates near the boundary. In turn, Theorem~\ref{thm:Fatou} is established 
using Theorem~\ref{thm:uniq-subcritical}. 

\medskip 

Pressing ahead, it is also worth contrasting the subcritical growth condition \eqref{subcritical} with the finiteness 
integral condition \eqref{u-integ}. Concretely, whenever $u\in\big[\mathcal{C}^\infty(\mathbb{R}^n_{+})\big]^M$ 
satisfies \eqref{subcritical} and $f:=u\big|^{{}^{\kappa-{\rm n.t.}}}_{\partial{\mathbb{R}}^n_{+}}$ 
exists at ${\mathcal{L}}^{n-1}$-a.e. point in ${\mathbb{R}}^{n-1}$ then necessarily $f$ is a 
locally bounded function; in fact, 
\begin{equation}\label{zxcvhsdr5-XXX}
\liminf_{\rho\to\infty}\rho^{-1}\|f\|_{[L^\infty(B_{n-1}(0',\rho))]^M}
\leq\liminf_{\rho\to\infty}\|u\|_{*,\rho}=0.
\end{equation}
As such, in the context of boundary value problems for the system $L$ in the upper half-space,  
the subcritical growth condition \eqref{subcritical} is most relevant whenever the formulation 
of the problem in question involves boundary data functions which are locally bounded (more 
precisely, satisfying the condition formulated in \eqref{zxcvhsdr5-XXX}). On the other hand, 
having a function $u\in\big[\mathcal{C}^\infty(\mathbb{R}^n_{+})\big]^M$ satisfying the finiteness integral 
condition \eqref{u-integ} and such that $f:=u\big|^{{}^{\kappa-{\rm n.t.}}}_{\partial{\mathbb{R}}^n_{+}}$ 
exists at ${\mathcal{L}}^{n-1}$-a.e. point in ${\mathbb{R}}^{n-1}$, guarantees that $f$ belongs to the weighted 
Lebesgue space $\Big[L^1\big({\mathbb{R}}^{n-1}\,,\,\tfrac{dx'}{1+|x'|^{n}}\big)\Big]^M$. This membership  
(which is the most general condition allowing one to define null-solutions to the system $L$ by taking 
the convolution with the Poisson kernel $P^L$ described in Theorem~\ref{thm:Poisson}) does not force 
$f$ to be locally bounded. 

Comparing the uniqueness statements from Theorem~\ref{thm:uniq-subcritical} and \eqref{Fatou-Uniqueness},
it is worth noting that the subcritical growth condition \eqref{subcritical} appearing in 
Theorem~\ref{thm:uniq-subcritical} decouples into \eqref{subcritical:mild} and 
\begin{equation}\label{semi-norms-eps.YYY}
\liminf_{\rho\to\infty}\Bigg[\rho^{-1}\sup_{\substack{0<t<\varepsilon\\ |x'|<\rho}}|u(x',t)|\Bigg]
=0\,\,\,\text{ for each fixed }\,\,\varepsilon>0.
\end{equation}
By way of contrast, in \eqref{Fatou-Uniqueness} in place of \eqref{semi-norms-eps.YYY} 
we are employing the finiteness integral condition \eqref{u-integ}.

In relation to the Fatou-type results discussed so far we wish to raise the following issue. 

\vskip 0.08in
{\bf Open Question~1.} 
{\it Can the format of the Fatou-type result from Theorem~\ref{thm:FP.111} be reconciled with that
of Theorem~\ref{thm:Fatou}? In other words, are these two seemingly distinct results particular manifestations
of a more general, inclusive phenomenon?}
\vskip 0.08in

Moving on, we say that a Lebesgue measurable function $f:\mathbb{R}^{n-1}\rightarrow\mathbb{C}$ 
belongs to the class of functions with subcritical growth, denoted ${\rm SCG}(\mathbb{R}^{n-1})$, 
provided 
\begin{equation}\label{SCG-f}
\int_{\mathbb{R}^{n-1}}\frac{|f(x')|}{1+|x'|^n}\,dx'<\infty\,\,\text{ and }\,\,
\lim_{\rho\to\infty}\Big[\rho^{-1}\|f\|_{L^\infty(B_{n-1}(0',\rho))}\Big]=0.
\end{equation}
As indicated in the corollary below (which appears in \cite{SCGC}), 
there is a Fatou-type result in the context of Theorem~\ref{thm:uniq-subcritical}
provided we slightly strengthen the condition demanded in \eqref{subcritical}. 

\begin{corollary}\label{thm:Fatou-SCG}
Let $L$ be an $M\times M$ homogeneous constant complex coefficient elliptic second-order system in ${\mathbb{R}}^n$. 
Assume $u\in\big[\mathcal{C}^\infty(\mathbb{R}^n_{+})\big]^M$ satisfies $Lu=0$ in $\mathbb{R}_{+}^n$, as well as 
the following Dini type condition at infinity:
\begin{equation}\label{41fgfgv}
\int_1^\infty\|u\|_{*,t}\,\frac{dt}{t}
=\int_1^\infty\Big(\sup_{B(0,t)\cap\mathbb{R}^n_{+}}|u|\Big)\,\frac{dt}{t^2}<\infty.
\end{equation}	

Then, for each aperture parameter $\kappa>0$, 
\begin{align}\label{Tafva.222234t5w5}
\begin{array}{l}
\big(u\big|^{{}^{\kappa-{\rm n.t.}}}_{\partial{\mathbb{R}}^n_{+}}\big)(x')
\,\,\text{ exists at ${\mathcal{L}}^{n-1}$-a.e. point }\,\,x'\in{\mathbb{R}}^{n-1},
\\[10pt]
\displaystyle
u\big|^{{}^{\kappa-{\rm n.t.}}}_{\partial{\mathbb{R}}^n_{+}}\,\,\text{ belongs to }\,\,
\big[{\rm SCG}({\mathbb{R}}^{n-1})\big]^M\subset\Big[L^1\Big({\mathbb{R}}^{n-1}\,,\,\frac{dx'}{1+|x'|^{n}}\Big)\Big]^M,
\\[12pt]
u(x',t)=\Big(P^L_t\ast\big(u\big|^{{}^{\kappa-{\rm n.t.}}}_{\partial{\mathbb{R}}^n_{+}}\big)\Big)(x')
\,\,\text{ for each point }\,\,(x',t)\in{\mathbb{R}}^n_{+}.
\end{array}
\end{align}
In particular, there exists a constant $C=C(L,\kappa)\in(0,\infty)$ for which  
the following Pointwise Maximum Principle holds:
\begin{equation}\label{Taf-UHN.ER.3}
\big|u\big|^{{}^{\kappa-{\rm n.t.}}}_{\partial{\mathbb{R}}^n_{+}}\big|
\leq{\mathcal{N}}_\kappa u\leq C{\mathcal{M}}\big(u\big|^{{}^{\kappa-{\rm n.t.}}}_{\partial{\mathbb{R}}^n_{+}}\big)
\,\,\text{ in }\,\,{\mathbb{R}}^{n-1}.
\end{equation}
\end{corollary}

Indeed, from \eqref{fvqafr}, \eqref{q34t3g3a}, \eqref{41fgfgv}, and Lebesgue’s Dominated Convergence Theorem 
it follows that \eqref{subcritical:mild} holds. Also, based on a dyadic decomposition argument
and \eqref{fvqafr} one can show that 
\begin{equation}\label{trrr}
\int_{\mathbb{R}^{n-1}}\frac{\sup_{0<t<1}|u(x',t)|}{1+|x'|^n}\,dx'
\leq C_n\int_1^\infty\|u\|_{*,t}\,\frac{dt}{t}.
\end{equation}
In view of \eqref{41fgfgv}, this means that \eqref{u-integ} holds. As a result, 
Theorem~\ref{thm:Fatou} 
applies and gives \eqref{Tafva.2222-ijk}. Together with the fact that the subcritical growth property is hereditary, 
we then conclude that all claims in \eqref{Tafva.222234t5w5} are true.

\section{Well-Posedness of Boundary Value Problems}
\label{S-3}

In this section we shall use the Poisson kernels and Fatou-type theorems from \S\ref{S-2} as tools 
for establishing the well-posedness of a variety of boundary value problems in the upper half-space 
${\mathbb{R}}^n_{+}$ for second-order, homogeneous, constant complex coefficient, elliptic systems in ${\mathbb{R}}^n$.

\subsection{The Dirichlet problem with boundary data from weighted $L^1$}
The template of the Fatou-type result from Theorem~\ref{thm:FP} prefigures the format of the well-posedness
result discussed in the theorem below. 

\begin{theorem}\label{Them-Gen}
Let $L$ be an $M\times M$ system with constant complex coefficients as in \eqref{L-def}-\eqref{L-ell.X}, 
and fix an aperture parameter $\kappa>0$. Then for each function 
\begin{align}\label{76tFfaf-7GF}
\begin{array}{c}
f:{\mathbb{R}}^{n-1}\to{\mathbb{C}}^{M}\,\,\text{ Lebesgue measurable}
\\[6pt]
\text{and }\,\,{\mathcal{M}}f\in L^1\big({\mathbb{R}}^{n-1}\,,\,\tfrac{dx'}{1+|x'|^{n-1}}\big)
\end{array}
\end{align}
{\rm (}recall that ${\mathcal{M}}$ is the Hardy-Littlewood maximal operator in ${\mathbb{R}}^{n-1}${\rm )} 
the boundary value problem
\begin{equation}\label{jk-lm-jhR-LLL-HM-RN.w.BVP}
\left\{
\begin{array}{l}
u\in\big[{\mathcal{C}}^{\infty}({\mathbb{R}}^n_{+})\big]^M,
\quad Lu=0\,\,\text{ in }\,\,{\mathbb{R}}^n_{+},
\\[8pt]
\displaystyle
\int_{\mathbb{R}^{n-1}}\big({\mathcal{N}}_{\kappa}u\big)(x')\,\frac{dx'}{1+|x'|^{n-1}}<\infty,
\\[12pt]
u\big|^{{}^{\kappa-{\rm n.t.}}}_{\partial{\mathbb{R}}^n_{+}}=f
\,\,\text{ at ${\mathcal{L}}^{n-1}$-a.e. point in }\,\,{\mathbb{R}}^{n-1},
\end{array}
\right.
\end{equation}
is uniquely solvable. Moreover, the solution $u$ of \eqref{jk-lm-jhR-LLL-HM-RN.w.BVP} 
is given by \eqref{exist:u} and satisfies
\begin{align}\label{jk-lm-jhR-LLL-HM-RN.w.BVP.2}
\int_{\mathbb{R}^{n-1}}\frac{|f(x')|}{1+|x'|^{n-1}}\,dx'
&\leq\int_{\mathbb{R}^{n-1}}\big({\mathcal{N}}_{\kappa}u\big)(x')\,\frac{dx'}{1+|x'|^{n-1}}
\nonumber\\[6pt]
&\leq C\int_{\mathbb{R}^{n-1}}\big({\mathcal{M}}f\big)(x')\,\frac{dx'}{1+|x'|^{n-1}}
\end{align}
for some constant $C=C(n,L,\kappa)\in(0,\infty)$ independent of $f$. 
\end{theorem}

For each $f$ as in \eqref{76tFfaf-7GF}, the membership of ${\mathcal{M}}f$ 
to $L^1\big({\mathbb{R}}^{n-1}\,,\,\tfrac{dx'}{1+|x'|^{n-1}}\big)$ implies 
that ${\mathcal{M}}f<\infty$ at ${\mathcal{L}}^{n-1}$-a.e. point in ${\mathbb{R}}^{n-1}$, 
which further entails $f\in\big[L^1_{\rm loc}({\mathbb{R}}^{n-1})\big]^M$. Granted this, 
Lebesgue's Differentiation Theorem applies and gives $|f|\leq{\mathcal{M}}f$ at 
${\mathcal{L}}^{n-1}$-a.e. point in ${\mathbb{R}}^{n-1}$. From this, the fact that 
$f$ is Lebesgue measurable, and the last property in \eqref{76tFfaf-7GF},  
we ultimately conclude that 
\begin{align}\label{76tFfaf-7GF.ewq}
f\in\Big[L^1\big({\mathbb{R}}^{n-1}\,,\,\tfrac{dx'}{1+|x'|^{n-1}}\big)\Big]^M
\subset\Big[L^1\big({\mathbb{R}}^{n-1}\,,\,\tfrac{dx'}{1+|x'|^{n}}\big)\Big]^M.
\end{align}
In particular, it is meaningful to define $u$ as in \eqref{exist:u}, and this ensures that 
the properties claimed in the first and last lines of \eqref{jk-lm-jhR-LLL-HM-RN.w.BVP} hold. 
Also, \eqref{nkc-EE-4} and \eqref{exist:Nu-Mf} imply \eqref{jk-lm-jhR-LLL-HM-RN.w.BVP.2} which, in turn, 
validates the finiteness condition in the second line of \eqref{jk-lm-jhR-LLL-HM-RN.w.BVP}.
This proves existence for the boundary value problem \eqref{jk-lm-jhR-LLL-HM-RN.w.BVP}, 
and uniqueness follows from Theorem~\ref{thm:FP}.

\subsection{The Dirichlet problem with data from $L^p$ and other related spaces}

The well-posedness of the $L^p$-Dirichlet boundary value problem was established in \cite{K-MMMM}.
As noted in \cite{SCGC}, our earlier results yield an alternative approach. 

\begin{theorem}\label{thm:Lp}
Let $L$ be an $M\times M$ homogeneous constant complex coefficient elliptic second-order system in ${\mathbb{R}}^n$ 
and fix an aperture parameter $\kappa>0$. For any $p\in(1,\infty)$ the $L^p$-Dirichlet boundary value problem for 
$L$ in $\mathbb{R}^{n}_{+}$, i.e., 
\begin{equation}\label{Dir-BVP-p}
\left\{
\begin{array}{l}
u\in\big[{\mathcal{C}}^\infty(\mathbb{R}^{n}_{+})\big]^M,\quad Lu=0\,\,\mbox{ in }\,\,\mathbb{R}^{n}_{+},
\\[4pt]
\mathcal{N}_\kappa u\,\,\text{ belongs to the space }\,\,L^p(\mathbb{R}^{n-1}),
\\[6pt]
u\big|^{{}^{\kappa-{\rm n.t.}}}_{\partial{\mathbb{R}}^{n}_{+}}=f
\,\,\text{ at ${\mathcal{L}}^{n-1}$-a.e. point in }\,\,{\mathbb{R}}^{n-1},
\end{array}
\right.
\end{equation}
has a unique solution for each $f\in\big[L^p(\mathbb{R}^{n-1})\big]^M$. 
Moreover, the solution $u$ of \eqref{Dir-BVP-p} is given by \eqref{exist:u} and satisfies 
\begin{equation}\label{eJHBbawvr}
\|f\|_{[L^p(\mathbb{R}^{n-1})]^M}\leq
\|\mathcal{N}_\kappa u\|_{L^p(\mathbb{R}^{n-1})}\leq C\|f\|_{[L^p(\mathbb{R}^{n-1})]^M}
\end{equation}
for some constant $C\in[1,\infty)$ that depends only on $L$, $n$, $p$, and $\kappa$.
\end{theorem}

Indeed, since 
\begin{align}\label{76tFfaf-7GF.abx}
L^p(\mathbb{R}^{n-1})\hookrightarrow L^1\big({\mathbb{R}}^{n-1}\,,\,\tfrac{dx'}{1+|x'|^{n-1}}\big)
\,\,\text{ for each }\,\,p\in[1,\infty),
\end{align}
and since the Hardy-Littlewood maximal operator 
\begin{align}\label{76tFfaf-7GF.aby}
{\mathcal{M}}:L^p({\mathbb{R}}^{n-1})\to L^p({\mathbb{R}}^{n-1})\,\,\text{ is bounded for each }\,\,p\in(1,\infty], 
\end{align}
we may regard \eqref{Dir-BVP-p} as a ``sub-problem'' of \eqref{jk-lm-jhR-LLL-HM-RN.w.BVP}.
As such, Theorem~\ref{Them-Gen} ensures existence (in the specified format) and uniqueness. 
The estimates claimed in \eqref{eJHBbawvr} are implied by \eqref{nkc-EE-4}, \eqref{exist:Nu-Mf}, 
and \eqref{76tFfaf-7GF.aby}.

\begin{remark}\label{ttFCCa}
A multitude of other important ``sub-problems'' of \eqref{jk-lm-jhR-LLL-HM-RN.w.BVP} present themselves. 
For example, if for each $p\in(1,\infty)$ and each Muckenhoupt weight $w\in A_p({\mathbb{R}}^{n-1})$ 
{\rm (}cf., e.g., \cite{GCRF85}{\rm )} we let $L^p_w(\mathbb{R}^{n-1})$ denote the space of Lebesgue 
measurable $p$-th power integrable functions in ${\mathbb{R}}^{n-1}$ with respect to the measure $w{\mathcal{L}}^{n-1}$, 
then the fact that (cf. \cite{K-MMMM})
\begin{align}\label{76tFfaf-7GF.abx.W}
\begin{array}{c}
L^p_w(\mathbb{R}^{n-1})\hookrightarrow L^1\big({\mathbb{R}}^{n-1}\,,\,\tfrac{dx'}{1+|x'|^{n-1}}\big)
\,\,\text{and}
\\[6pt]
{\mathcal{M}}:L^p_w({\mathbb{R}}^{n-1})\to L^p_w({\mathbb{R}}^{n-1})\,\,\text{ boundedly},
\end{array}
\end{align}
ultimately implies that for each integrability exponent $p\in(1,\infty)$, each weight $w\in A_p({\mathbb{R}}^{n-1})$, 
and each aperture parameter $\kappa>0$, the $L^p_w$-Dirichlet boundary value problem for $L$ in $\mathbb{R}^{n}_{+}$, i.e., 
\begin{equation}\label{Dir-BVP-p-WWW}
\left\{
\begin{array}{l}
u\in\big[{\mathcal{C}}^\infty(\mathbb{R}^{n}_{+})\big]^M,\quad Lu=0\,\,\mbox{ in }\,\,\mathbb{R}^{n}_{+},
\\[4pt]
\mathcal{N}_\kappa u\,\,\text{ belongs to the space }\,\,L^p_w(\mathbb{R}^{n-1}),
\\[6pt]
u\big|^{{}^{\kappa-{\rm n.t.}}}_{\partial{\mathbb{R}}^{n}_{+}}=f
\,\,\text{ at ${\mathcal{L}}^{n-1}$-a.e. point in }\,\,{\mathbb{R}}^{n-1},
\end{array}
\right.
\end{equation}
has a unique solution for each $f\in\big[L^p_w(\mathbb{R}^{n-1})\big]^M$, and the solution $u$ of \eqref{Dir-BVP-p} 
{\rm (}which continues to be given by \eqref{exist:u}{\rm )} satisfies 
\begin{equation}\label{eJHBbawvr-EW}
\|f\|_{[L^p_w(\mathbb{R}^{n-1})]^M}\leq
\|\mathcal{N}_\kappa u\|_{L^p_w(\mathbb{R}^{n-1})}\leq C\|f\|_{[L^p_w(\mathbb{R}^{n-1})]^M}.
\end{equation}
Similarly, since for the Lorentz spaces $L^{p,q}({\mathbb{R}}^{n-1})$ with $p\in(1,\infty)$,
$q\in(0,\infty]$, we also have (again, see \cite{K-MMMM})
\begin{align}\label{76tFfaf-7GF.abx.L}
\begin{array}{c}
L^{p,q}(\mathbb{R}^{n-1})\hookrightarrow L^1\big({\mathbb{R}}^{n-1}\,,\,\tfrac{dx'}{1+|x'|^{n-1}}\big)
\,\,\text{and}
\\[6pt]
{\mathcal{M}}:L^{p,q}({\mathbb{R}}^{n-1})\to L^{p,q}({\mathbb{R}}^{n-1})\,\,\text{ boundedly},
\end{array}
\end{align}
we also conclude that the version of the Dirichlet problem \eqref{Dir-BVP-p} naturally formulated in such a setting 
continues to be well-posed. To offer yet another example, recall the scale of Morrey spaces 
$\mathfrak{L}^{p,\lambda}({\mathbb{R}}^{n-1})$ in ${\mathbb{R}}^{n-1}$, 
defined for each $p\in(1,\infty)$ and $\lambda\in(0,n-1)$ according to 
\begin{equation}\label{MOR.1}
\mathfrak{L}^{p,\lambda}({\mathbb{R}}^{n-1}):=\Big\{f\in L^p_{\rm loc}({\mathbb{R}}^{n-1}):\,
\|f\|_{\mathfrak{L}^{p,\lambda}({\mathbb{R}}^{n-1})}<\infty\Big\}
\end{equation}
where
\begin{equation}\label{MOR.2}
\|f\|_{\mathfrak{L}^{p,\lambda}({\mathbb{R}}^{n-1})}:=
\sup_{x'\in{\mathbb{R}}^{n-1},\,r>0}\Big(r^{-\lambda}\int_{B_{n-1}(x',r)}|f|^p\,d{\mathcal{L}}^{n-1}\Big)^{1/p}.
\end{equation}
Given that 
\begin{align}\label{MOR.9}
\begin{array}{c}
\mathfrak{L}^{p,\lambda}({\mathbb{R}}^{n-1})\subset
L^1\big({\mathbb{R}}^{n-1}\,,\,\tfrac{dx'}{1+|x'|^{n-1}}\big)
\\[6pt]
\text{provided }\,\,1<p<\infty\,\,\text{ and }\,\,0<\lambda<n-1,
\end{array}
\end{align}
and since (cf., e.g., \cite{CF}) 
\begin{align}\label{MOR.11}
\parbox{8.0cm}{the Hardy-Littlewood operator ${\mathcal{M}}$ is bounded on 
$\mathfrak{L}^{p,\lambda}({\mathbb{R}}^{n-1})$ if $1<p<\infty$ and $0<\lambda<n-1$,}
\end{align}
we once again conclude that the version of the Dirichlet problem \eqref{Dir-BVP-p} 
naturally formulated in terms of Morrey spaces becomes well-posed. For more examples 
of this nature and further details the reader is referred to \cite{K-MMMM}.
\end{remark}

Later on, in Theorem~\ref{thm:Linfty}, we shall see that in fact the end-point $p=\infty$ is permissible 
in the context of Theorem~\ref{thm:Lp}; that is, the $L^\infty$-Dirichlet problem is well-posed. At the other 
end of the spectrum, i.e., for $p=1$, the very nature of \eqref{Dir-BVP-p} changes. Indeed, at least when 
$L=\Delta$, the Laplacian in ${\mathbb{R}}^n$, from \cite[Proposition~1, p.\,119]{Stein93} we know that 
for any harmonic function $u$ in $\mathbb{R}^{n}_{+}$ with $\mathcal{N}_\kappa u\in L^1(\mathbb{R}^{n-1})$ there 
exists $f\in H^1({\mathbb{R}}^{n-1})$ such that $u(x',t)=(P^\Delta_t\ast f)(x')$ for each $(x',t)\in{\mathbb{R}}^n_{+}$.
In concert with Theorem~\ref{thm:FP} and the observation that $H^1({\mathbb{R}}^{n-1})$ is a subspace of 
$L^1({\mathbb{R}}^{n-1})$, this implies that any harmonic function $u$ in $\mathbb{R}^{n}_{+}$ with 
$\mathcal{N}_\kappa u\in L^1(\mathbb{R}^{n-1})$ (for some $\kappa>0$) has a nontangential boundary trace 
$u\big|^{{}^{\kappa-{\rm n.t.}}}_{\partial{\mathbb{R}}^{n}_{+}}$ at ${\mathcal{L}}^{n-1}$-a.e. 
point in ${\mathbb{R}}^{n-1}$ which actually belongs to the Hardy space $H^1(\mathbb{R}^{n-1})$.
Thus, the boundary data are necessarily in a Hardy space in this case. This feature accounts for the 
manner in which we now formulate the following well-posedness result. 

\begin{theorem}\label{thm:Lp-H111}
Let $L$ be an $M\times M$ homogeneous constant complex coefficient elliptic second-order system in ${\mathbb{R}}^n$
and fix an aperture parameter $\kappa>0$. Then the $(H^1,L^1)$-Dirichlet boundary value problem for $L$ in 
$\mathbb{R}^{n}_{+}$, i.e., 
\begin{equation}\label{Dir-BVP-p.H111}
\left\{
\begin{array}{l}
u\in\big[{\mathcal{C}}^\infty(\mathbb{R}^{n}_{+})\big]^M,\quad Lu=0\,\,\mbox{ in }\,\,\mathbb{R}^{n}_{+},
\\[4pt]
\mathcal{N}_\kappa u\,\,\text{ belongs to the space }\,\,L^1(\mathbb{R}^{n-1}),
\\[6pt]
u\big|^{{}^{\kappa-{\rm n.t.}}}_{\partial{\mathbb{R}}^{n}_{+}}=f
\,\,\text{ at ${\mathcal{L}}^{n-1}$-a.e. point in }\,\,{\mathbb{R}}^{n-1},
\end{array}
\right.
\end{equation}
has a unique solution for each $f$ belonging to the Hardy space $\big[H^1(\mathbb{R}^{n-1})\big]^M$. 
In addition, the solution $u$ of \eqref{Dir-BVP-p.H111} is given by \eqref{exist:u} and satisfies 
\begin{equation}\label{eJHBbawvr.H111}
\|\mathcal{N}_\kappa u\|_{L^1(\mathbb{R}^{n-1})}\leq C\|f\|_{[H^1(\mathbb{R}^{n-1})]^M}
\end{equation}
for some constant $C\in(0,\infty)$ which depends only on $L$, $n$, and $\kappa$.
\end{theorem}

Theorem~\ref{thm:Lp-H111} has been originally established in \cite{K-MMMM}, and the present work yields 
an alternative proof. Indeed, existence follows from item {\it (e)} of Theorem~\ref{thm:Poisson.II}, 
while uniqueness is implied by Theorem~\ref{thm:FP}. 

In relation to the work discussed so far in this section we wish to formulate several open questions. 
We start by formulating a question which asks for allowing more general operators in 
the statement of \cite[Proposition~1, p.\,119]{Stein93}.

\vskip 0.08in
{\bf Open Question~2.} 
{\it Let $L$ be an $M\times M$ homogeneous constant complex coefficient elliptic second-order system 
in ${\mathbb{R}}^n$. Suppose $0<p\leq\infty$ and fix some $\kappa>0$. Also, consider a function 
$u\in\big[\mathcal{C}^\infty(\mathbb{R}^n_{+})\big]^M$ satisfying $Lu=0$ in $\mathbb{R}_{+}^n$. 
Show that $\mathcal{N}_\kappa u\in L^p(\mathbb{R}^{n-1})$ if and only if there exists 
$f\in\big[H^p(\mathbb{R}^{n-1})\big]^M$ such that $u(x',t)=(P^L_t\ast f)(x')$ for each $(x',t)\in{\mathbb{R}}^n_{+}$.
Moreover, show that $\big\|\mathcal{N}_\kappa u\big\|_{L^p(\mathbb{R}^{n-1})}\approx\|f\|_{[H^p(\mathbb{R}^{n-1})]^M}$.}
\vskip 0.08in

Theorem~\ref{thm:Lp} provides an answer to this question in the range $p\in(1,\infty)$, while Theorem~\ref{thm:Linfty}
(discussed later on) addresses the case $p=\infty$. Also, item {\it (e)} of Theorem~\ref{thm:Poisson.II} is directly 
relevant to the issue at hand in the range $p\in\big(\tfrac{n-1}{n}\,,\,1\big]$. 

Our next question asks for allowing more general operators in the formulation of 
\cite[Theorem~4.23, p.\,190]{GCRF85}.

\vskip 0.08in
{\bf Open Question~3.} 
{\it Let $L$ be an $M\times M$ homogeneous constant complex coefficient elliptic second-order system 
in ${\mathbb{R}}^n$. Suppose $0<p\leq\infty$ and fix some $\kappa>0$. Also, consider a function 
$u\in\big[\mathcal{C}^\infty(\mathbb{R}^n_{+})\big]^M$ satisfying $Lu=0$ in $\mathbb{R}_{+}^n$ and
$\mathcal{N}_\kappa u\in L^p(\mathbb{R}^{n-1})$. Show that $f:=\lim\limits_{t\to 0^{+}}u(\cdot,t)$ exists 
in the sense of tempered distributions in ${\mathbb{R}}^{n-1}$, i.e., in $\big[{\mathcal{S}}'({\mathbb{R}}^{n-1})\big]^M$.}  
\vskip 0.08in

The following question pertains to the well-posedness of a brand of Dirichlet problem in which the boundary trace
is taken in a weak, distributional sense. 

\vskip 0.08in
{\bf Open Question~4.} 
{\it Let $L$ be an $M\times M$ homogeneous constant complex coefficient elliptic second-order system 
in ${\mathbb{R}}^n$. Suppose $0<p\leq\infty$ and fix some $\kappa>0$. Show that for each 
$f\in\big[H^p({\mathbb{R}}^{n-1})\big]^M$ the following boundary value problem is uniquely 
solvable and a naturally accompanying estimate holds:}
\begin{equation}\label{Dir-BVP-p.hphp}
\left\{
\begin{array}{l}
u\in\big[{\mathcal{C}}^\infty(\mathbb{R}^{n}_{+})\big]^M,\quad
Lu=0\,\,\mbox{ in }\,\,\mathbb{R}^{n}_{+},
\\[4pt]
\mathcal{N}_\kappa u\,\,\text{ belongs to the space }\,\,L^p(\mathbb{R}^{n-1}),
\\[6pt]
\lim\limits_{t\to 0^{+}}u(\cdot,t)=f\,\,\text{ in }\,\,\big[{\mathcal{S}}'({\mathbb{R}}^{n-1})\big]^M.
\end{array}
\right.
\end{equation}
Our earlier work shows that \eqref{Dir-BVP-p.hphp} is indeed well-posed if $p\in[1,\infty]$. 

The question below has to do with the solvability of the so-called Regularity problem.
This is a brand of Dirichlet problem in which the boundary data is selected from Sobolev 
spaces ($L^p$-based, of order one) and, as a result, stronger regularity is demanded of the solution. 

\vskip 0.08in
{\bf Open Question~5.} 
{\it Let $L$ be an $M\times M$ homogeneous constant complex coefficient elliptic second-order system 
in ${\mathbb{R}}^n$. Fix an integrability exponent $p\in(1,\infty)$ along with an aperture parameter 
$\kappa>0$. Also, pick an arbitrary $f$ in the Sobolev space $\big[W^{1,p}({\mathbb{R}}^{n-1})\big]^M$.
Find additional conditions, either on the system $L$ or the boundary datum $f$, guaranteeing that
the Regularity problem formulated as} 
\begin{equation}\label{Dir-REG-p.}
\left\{
\begin{array}{l}
u\in\big[{\mathcal{C}}^\infty(\mathbb{R}^{n}_{+})\big]^M,\quad
Lu=0\,\,\mbox{ in }\,\,\mathbb{R}^{n}_{+},
\\[4pt]
\mathcal{N}_\kappa u,\,\mathcal{N}_\kappa(\nabla u)\,\,\text{ belong to }\,\,L^p(\mathbb{R}^{n-1}),
\\[6pt]
u\big|^{{}^{\kappa-{\rm n.t.}}}_{\partial{\mathbb{R}}^n_{+}}=f
\,\,\text{ at ${\mathcal{L}}^{n-1}$-.a.e. point in }\,\,{\mathbb{R}}^{n-1},
\end{array}
\right.
\end{equation}
{\it is solvable and a naturally accompanying estimate holds.}
\vskip 0.08in

Work relevant to this question may be found in \cite{H-MMMM} where a large class of systems $L$,  
including scalar operators (such as the Laplacian) as well as the Lam\'e system \eqref{TYd-YG-76g}, has been
identified with the property that the Regularity problem \eqref{Dir-REG-p.} is uniquely solvable for each 
$f\in\big[W^{1,p}({\mathbb{R}}^{n-1})\big]^M$ with $1<p<\infty$. Also, in \cite{S-MMMM} the following 
link between the solvability of the Regularity problem \eqref{Dir-REG-p.}, and the domain of the infinitesimal 
generator of the $C_0$-semigroup $T=\{T(t)\}_{t\geq 0}$ associated with $L$ as in \eqref{eq:Taghb8}, 
has been established.

\begin{theorem}\label{V-Naa.11}
Let $L$ be an $M\times M$ homogeneous constant complex coefficient elliptic second-order system 
in ${\mathbb{R}}^n$. Fix some $p\in(1,\infty)$ and consider the $C_0$-semigroup $T=\{T(t)\}_{t\geq 0}$ 
on $\big[L^p({\mathbb{R}}^{n-1})\big]^M$ associated with $L$ as in \eqref{eq:Taghb8}. Denote by ${\mathbf{A}}$ 
the infinitesimal generator of $T$, with domain $D({\mathbf{A}})$. Then $D({\mathbf{A}})$ is a dense linear 
subspace of $\big[W^{1,p}({\mathbb{R}}^{n-1})\big]^M$ and, in fact,
\begin{align}\label{eq:tfc.1-new}
D({\mathbf{A}})=\big\{f\in\big[W^{1,p}({\mathbb{R}}^{n-1})\big]^M:\,&
\text{the problem \eqref{Dir-REG-p.} with} 
\nonumber\\[-2pt]
&\text{boundary datum $f$ is solvable}\big\}.
\end{align}
In particular, $D({\mathbf{A}})=\big[W^{1,p}({\mathbb{R}}^{n-1})\big]^M$ if and only if 
the Regularity problem \eqref{Dir-REG-p.} is solvable for arbitrary data $f\in\big[W^{1,p}({\mathbb{R}}^{n-1})\big]^M$.
\end{theorem}

Moving on, we present a Fatou-type theorem from \cite{SCGC} which refines work in 
\cite[Theorem~6.1]{K-MMMM} and \cite[Corollary~6.3]{K-MMMM}.

\begin{theorem}\label{TFac-gR}
Let $L$ be an $M\times M$ homogeneous constant complex coefficient elliptic second-order system 
in ${\mathbb{R}}^n$. Assume that $u\in\big[\mathcal{C}^\infty(\mathbb{R}^n_{+})\big]^M$ satisfies 
$Lu=0$ in $\mathbb{R}_{+}^n$. If $\mathcal{N}_\kappa u\in L^p(\mathbb{R}^{n-1})$ for some $p\in[1,\infty]$
and $\kappa>0$, then
\begin{align}\label{Tafva.222fawcr}
\begin{array}{l}
\big(u\big|^{{}^{\kappa-{\rm n.t.}}}_{\partial{\mathbb{R}}^n_{+}}\big)(x')
\,\,\text{ exists at ${\mathcal{L}}^{n-1}$-a.e. point }\,\,x'\in{\mathbb{R}}^{n-1},
\\[10pt]
\displaystyle
u\big|^{{}^{\kappa-{\rm n.t.}}}_{\partial{\mathbb{R}}^n_{+}}\,\,\text{ belongs to }\,\,
\big[L^p({\mathbb{R}}^{n-1})\big]^M\subset\Big[L^1\Big({\mathbb{R}}^{n-1}\,,\,\frac{dx'}{1+|x'|^{n}}\Big)\Big]^M,
\\[12pt]
u(x',t)=\Big(P^L_t\ast\big(u\big|^{{}^{\kappa-{\rm n.t.}}}_{\partial{\mathbb{R}}^n_{+}}\big)\Big)(x')
\,\,\text{ for each point }\,\,(x',t)\in{\mathbb{R}}^n_{+}.
\end{array}
\end{align}
\end{theorem}
 
In view of \eqref{76tFfaf-7GF.abx}, when $p\in[1,\infty)$ all claims are direct consequences 
of Theorem~\ref{thm:FP} (also bearing \eqref{nkc-EE-4} in mind). The result corresponding to 
the end-point $p=\infty$ is no longer implied by Theorem~\ref{thm:FP} as the finiteness condition 
in the second line of \eqref{jk-lm-jhR-LLL-HM-RN.w} fails in general for bounded functions
(the best one can say in such a scenario is that $\mathcal{N}_\kappa u\in L^\infty(\mathbb{R}^{n-1})$).
Nonetheless, Corollary~\ref{thm:Fatou-SCG} applies and all desired conclusions now follow from this. 

As a corollary, we note that, given any aperture parameter $\kappa>0$ along with an integrability 
exponent $p\in(1,\infty]$, Theorem~\ref{TFac-gR} implies (together with \eqref{nkc-EE-4}, \eqref{exist:Nu-Mf},
and \eqref{76tFfaf-7GF.aby}) the following $L^p$-styled Maximum Principle:
\begin{align}\label{Ta-jy6GGa-yT}
\begin{array}{c}
\big\|u\big|^{{}^{\kappa-{\rm n.t.}}}_{\partial{\mathbb{R}}^n_{+}}\big\|_{[L^p({\mathbb{R}}^{n-1})]^M}
\approx\big\|{\mathcal{N}}_\kappa u\big\|_{L^p({\mathbb{R}}^{n-1})}\,\,\text{ uniformly in} 
\\[6pt]
\text{the class of functions $u\in\big[\mathcal{C}^\infty(\mathbb{R}^n_{+})\big]^M$ satisfying} 
\\[6pt]
\text{$Lu=0$ in $\mathbb{R}_{+}^n$ as well as $\mathcal{N}_\kappa u\in L^p(\mathbb{R}^{n-1})$.}
\end{array}
\end{align}

Theorem~\ref{TFac-gR} is sharp, in the sense that the corresponding result fails for 
$p\in\big(\tfrac{n-1}{n}\,,\,1\big)$. To see that this is the case, fix some vector 
$a\in{\mathbb{C}}^M\setminus\{0\}$
along with some point $z'\in{\mathbb{R}}^{n-1}\setminus\{0'\}$ and consider the function 
\begin{equation}\label{FCT-TR.nnn}
u_\star(x',t):=K^L(x',t)a-K^L(x'-z',t)a\,\,\text{ for each }\,\,(x',t)\in\mathbb{R}^n_{+}.
\end{equation}
Then $u_\star$ belongs to the space
$\big[\mathcal{C}^{\infty}(\overline{\mathbb{R}^n_{+}}\setminus\{(0',0),(z',0)\})\big]^M$, 
satisfies $Lu_\star=0$ in $\mathbb{R}_{+}^n$, and 
$\Big(u_\star\big|^{{}^{\kappa-{\rm n.t.}}}_{\partial{\mathbb{R}}^{n}_{+}}\Big)(x')=0$ 
for every aperture parameter $\kappa>0$ and every point 
$x'\in{\mathbb{R}}^{n-1}\setminus\{0',z'\}$. In addition, 
we may choose $a,z'$ such that $u_\star$ is not identically zero 
(otherwise this would force $K^L(x',t)$ to be independent of $x'$, 
a happenstance precluded by, e.g., \eqref{eq:Kest} and \eqref{eq:IG6gy.2}-\eqref{eq:Gvav7g5}).
Hence, on the one hand, the Poisson integral representation formula in the last line of 
\eqref{Tafva.222fawcr} presently fails. On the other hand, from the well-known fact that 
\begin{equation}\label{Ka-jGG.1}
\delta_{0'}-\delta_{z'}\in H^p({\mathbb{R}}^{n-1})\,\,\text{ for each }\,\,p\in\big(\tfrac{n-1}{n}\,,\,1\big),
\end{equation}
it follows that 
\begin{equation}\label{Ka-jGG.2}
f:=(\delta_{0'}-\delta_{z'})a\in\big[H^p({\mathbb{R}}^{n-1})\big]^M
\,\,\text{ for each }\,\,p\in\big(\tfrac{n-1}{n}\,,\,1\big).
\end{equation}
Moreover, $f$ is related to the function $u$ from \eqref{FCT-TR.nnn} via 
$u_\star(x',t)=(P^L_t\ast f)(x')$ at each point $(x',t)\in{\mathbb{R}}^n_{+}$, with the convolution 
understood as in \eqref{exist:u-123}. As such, \eqref{exist:Nu-Mf-Hp} implies that for each 
aperture parameter $\kappa>0$ we have 
\begin{equation}\label{exist:Nu-Mf-Hp.iii}
\mathcal{N}_\kappa u_\star\in L^p(\mathbb{R}^{n-1})
\,\,\text{ for each }\,\,p\in\big(\tfrac{n-1}{n}\,,\,1\big).
\end{equation}
Parenthetically, we wish to pint out that the membership in \eqref{exist:Nu-Mf-Hp.iii} may also be 
justified directly based on \eqref{FCT-TR.nnn} and the estimates for the kernel function $K^L$ from 
item {\it (a)} in Theorem~\ref{thm:Poisson.II} which, collectively, show that  
\begin{equation}\label{eq:IG6gy.2-LaIP}
\parbox{10.60cm}{for $x'\in{\mathbb{R}}^{n-1}$, the nontangential maximal function 
$\big(\mathcal{N}_\kappa u_\star\big)(x')$ behaves like $|x'|^{1-n}$ if $x'$ is near $0'$, 
like $|x'-z'|^{1-n}$ if $x'$ is near $z'$, like $|x'|^{-n}$ if $x'$ is near infinity, 
and is otherwise bounded.}
\end{equation}
Granted this, it follows that $\mathcal{N}_\kappa u_\star\in L^p(\mathbb{R}^{n-1})$ if and only if 
$p(n-1)<n-1$ and $pn>n-1$, a set of conditions equivalent to $p\in\big(\tfrac{n-1}{n}\,,\,1\big)$.

To summarize, the function $u_\star$ defined in \eqref{FCT-TR.nnn} satisfies, for each aperture parameter $\kappa>0$,
\begin{equation}\label{rt-vba-jg}
\left\{
\begin{array}{l}
u_\star\in\big[{\mathcal{C}}^{\infty}({\mathbb{R}}^n_{+})\big]^M,\quad Lu_\star=0\,\,\text{ in }\,\,{\mathbb{R}}^n_{+},
\\[6pt]
{\mathcal{N}}_{\kappa}u_\star\in L^p(\mathbb{R}^{n-1})\,\,\text{ for each }\,\,p\in\big(\tfrac{n-1}{n}\,,\,1\big),
\\[8pt]
u_\star\big|^{{}^{\kappa-{\rm n.t.}}}_{\partial{\mathbb{R}}^{n}_{+}}=0
\,\,\text{ at ${\mathcal{L}}^{n-1}$-a.e. point in }\,\,{\mathbb{R}}^{n-1},
\\[8pt]
\text{the function $u_\star$ is not identically zero in ${\mathbb{R}}^n_{+}$.}
\end{array}
\right.
\end{equation}
This is in sharp contrast to Theorem~\ref{TFac-gR}, and points to the fact that when $p<1$ 
the pointwise nontangential boundary trace of a null-solution of the system $L$ no longer 
characterizes the original function.

\subsection{The subcritical growth Dirichlet problem}\label{section:SCG-BVP}

Recall that ${\rm SCG}(\mathbb{R}^{n-1})$ stands for the class of functions 
exhibiting subcritical growth in $\mathbb{R}^{n-1}$, defined as in \eqref{SCG-f}.
In relation to this class, we have the following well-posedness result from \cite{SCGC}. 

\begin{theorem}\label{thm:SCG}
Let $L$ be an $M\times M$ homogeneous constant complex coefficient elliptic second-order system in 
${\mathbb{R}}^n$. Also, fix an aperture parameter $\kappa>0$. Then the subcritical growth Dirichlet 
boundary value problem for $L$ in $\mathbb{R}^{n}_{+}$, formulated as
\begin{equation}\label{Dir-BVP-SCG}
\left\{
\begin{array}{l}
u\in\big[{\mathcal{C}}^\infty(\mathbb{R}^{n}_{+})\big]^M,\quad Lu=0\,\,\mbox{ in }\,\,\mathbb{R}^{n}_{+},
\\[4pt]
\lim\limits_{\rho\to\infty}\|u\|_{*,\rho}=0,
\\[8pt]
u\big|^{{}^{\kappa-{\rm n.t.}}}_{\partial{\mathbb{R}}^{n}_{+}}=f
\,\,\text{ at ${\mathcal{L}}^{n-1}$-a.e. point in }\,\,{\mathbb{R}}^{n-1},
\end{array}
\right.
\end{equation}
has a unique solution for each $f\in\big[{\rm SCG}(\mathbb{R}^{n-1})\big]^M$. 
Moreover, the solution $u$ of \eqref{Dir-BVP-SCG} is given by \eqref{exist:u} and satisfies 
the following Weak Local Maximum Principle: 
\begin{equation}\label{q43t3g}
\sup_{B(0,\rho)\cap\mathbb{R}^n_{+}}|u|\leq C\,\Bigg(\|f\|_{[L^\infty(B_{n-1}(0',2\rho))]^M}
+\int_{\mathbb{R}^{n-1}\setminus B_{n-1}(0',2\rho)}\frac{\rho|f(y')|}{\rho^n+|y'|^n}\,dy'\Bigg)
\end{equation}
for each $\rho\in(0,\infty)$, where $C\in[1,\infty)$ depends only on $L$ and $n$.
\end{theorem}

Note that having $f\in\big[{\rm SCG}(\mathbb{R}^{n-1})\big]^M$ ensures that $f$ satisfies \eqref{exist:f} 
which, in turn, allows us to define the solution $u$ via the convolution with the Poisson kernel (cf. 
item {\it (c)} of Theorem~\ref{thm:Poisson.II}). Uniqueness follows at once from Theorem~\ref{thm:uniq-subcritical}.
To close, we remark that the second condition imposed on $f$ in \eqref{SCG-f}, 
which amounts to saying that $f$ has subcritical growth, is natural in the context of \eqref{Dir-BVP-SCG}.
Indeed, whenever $u$ satisfies $\lim\limits_{\rho\to\infty}\|u\|_{*,\rho}=0$ and 
$f:=u\big|^{{}^{\kappa-{\rm n.t.}}}_{\partial{\mathbb{R}}^{n}_{+}}$ exists ${\mathcal{L}}^{n-1}$-a.e. in 
${\mathbb{R}}^{n-1}$ it is not difficult to see that 
\begin{equation}\label{zxcvhsdr5}
\rho^{-1}\|f\|_{[L^\infty(B_{n-1}(0',\rho))]^M}\leq\|u\|_{*,\rho}\to 0\,\,\text{ as }\,\,\rho\to\infty,
\end{equation}
which ultimately implies the second condition in \eqref{SCG-f}.

\subsection{The $L^\infty$-Dirichlet boundary value problem}

Here we revisit Theorem~\ref{thm:Lp} and consider the (initially forbidden) end-point $p=\infty$.
Our result below is well-known in the particular case when $L=\Delta$, the Laplacian in ${\mathbb{R}}^n$, 
but all known proofs (e.g., that of \cite[Theorem~4.8, p.\,174]{GCRF85}, or that of \cite[Proposition~1, p.\,199]{St70}) 
make use of specialized properties of harmonic functions. Following \cite{SCGC}, here we are able to 
treat the $L^\infty$-Dirichlet boundary value problem in $\mathbb{R}^{n}_{+}$ for any homogeneous constant 
complex coefficient elliptic second-order system in a conceptually simple manner, relying on our more 
general result from Theorem~\ref{thm:SCG}. 

\begin{theorem}\label{thm:Linfty}
Let $L$ be an $M\times M$ homogeneous constant complex coefficient elliptic second-order system in ${\mathbb{R}}^n$ 
and fix an aperture parameter $\kappa>0$. Then the $L^\infty$-Dirichlet boundary value problem for $L$ in 
$\mathbb{R}^{n}_{+}$,
\begin{equation}\label{Dir-BVP-Linfty}
\left\{
\begin{array}{l}
u\in\big[{\mathcal{C}}^\infty(\mathbb{R}^{n}_{+})\cap L^\infty(\mathbb{R}^n_{+})\big]^M,
\\[4pt]
Lu=0\,\,\mbox{ in }\,\,\mathbb{R}^{n}_{+},
\\[6pt]
u\big|^{{}^{\kappa-{\rm n.t.}}}_{\partial{\mathbb{R}}^{n}_{+}}=f
\,\,\text{ at ${\mathcal{L}}^{n-1}$-a.e. point in }\,\,{\mathbb{R}}^{n-1},
\end{array}
\right.
\end{equation}
has a unique solution for each $f\in\big[L^\infty(\mathbb{R}^{n-1})\big]^M$. 
Moreover, the solution $u$ of \eqref{Dir-BVP-Linfty} is given by \eqref{exist:u} 
and satisfies the Weak Maximum Principle
\begin{equation}\label{eJHBb}
\|f\|_{[L^\infty(\mathbb{R}^{n-1})]^M}\leq
\|u\|_{[L^\infty(\mathbb{R}^n_{+})]^M}\leq C\|f\|_{[L^\infty(\mathbb{R}^{n-1})]^M},
\end{equation}
for some constant $C\in[1,\infty)$ that depends only on $L$ and $n$.
\end{theorem}

Since 
\begin{align}\label{76tFfaf-7GF.ab345}
L^\infty(\mathbb{R}^{n-1})\subset{\rm SCG}(\mathbb{R}^{n-1}),
\end{align}
and since \eqref{q43t3g} readily implies \eqref{eJHBb}, we may regard \eqref{Dir-BVP-Linfty} 
as a ``sub-problem'' of \eqref{Dir-BVP-SCG}. This ensures existence (in the specified format), 
uniqueness, as well as the estimate claimed in \eqref{eJHBb}.

\subsection{The classical Dirichlet boundary value problem}

Given $E\subseteq{\mathbb{R}}^m$, for some $m\in{\mathbb{N}}$, define ${\mathcal{C}}^0_b(E)$ 
to be the space of ${\mathbb{C}}$-valued functions defined on $E$ which are continuous and bounded. 
The theorem below appears in \cite{SCGC}. The particular case when $L=\Delta$, the Laplacian in ${\mathbb{R}}^n$ 
is a well-known, classical result (see, e.g., \cite[Theorem~7.5, p.\,148]{ABR}, or \cite[Theorem~4.4, p.\,170]{GCRF85}), 
so the novelty here is the consideration of much more general operators. 

\begin{theorem}\label{thm:CLASSICAL}
Let $L$ be an $M\times M$ homogeneous constant complex coefficient elliptic second-order system in ${\mathbb{R}}^n$. 
Then the classical Dirichlet boundary value problem for $L$ in $\mathbb{R}^{n}_{+}$,
\begin{equation}\label{Dir-BVP-CLASSICAL}
\left\{
\begin{array}{l}
u\in\big[{\mathcal{C}}^\infty(\mathbb{R}^{n}_{+})\cap{\mathcal{C}}^0_b(\overline{\mathbb{R}^{n}_{+}})\big]^M,
\\[4pt]
Lu=0\,\,\mbox{ in }\,\,\mathbb{R}^{n}_{+},
\\[4pt]
u\big|_{\partial{\mathbb{R}}^{n}_{+}}=f\,\,\text{ in }\,\,{\mathbb{R}}^{n-1},
\end{array}
\right.
\end{equation}
has a unique solution for each $f\in\big[{\mathcal{C}}^0_b(\mathbb{R}^{n-1})\big]^M$. 
Moreover, the solution $u$ of \eqref{Dir-BVP-CLASSICAL} is given by \eqref{exist:u} 
and satisfies the Weak Maximum Principle
\begin{equation}\label{eJHBb-CLASSICAL}
\sup_{\mathbb{R}^{n-1}}|f|\leq\sup_{\overline{\mathbb{R}^n_{+}}}|u|\leq C\sup_{\mathbb{R}^{n-1}}|f|
\end{equation}
for some constant $C\in[1,\infty)$ that depends only on $L$ and $n$.
\end{theorem}

Existence is a consequence of item {\it (c)} of Theorem~\ref{thm:Poisson.II}, uniqueness is implied by 
Theorem~\ref{thm:Linfty}, and \eqref{eJHBb-CLASSICAL} follows from \eqref{eJHBb}. 

The nature of the constant $C$ appearing in the Weak Maximum Principle \eqref{eJHBb-CLASSICAL} 
(as well as other related inequalities) has been studied by G.~Kresin and V.~Maz'ya in \cite{KrMa}.

\vskip 0.08in
{\bf Open Question~6.} 
{\it In the context of Theorem~\ref{thm:CLASSICAL}, if the boundedness requirement is dropped 
both for the boundary datum and for the solution, does the resulting boundary value problem 
continue to be solvable?} 
\vskip 0.08in

This is known to be the case when $L=\Delta$, the Laplacian in ${\mathbb{R}}^n$; see, e.g. 
\cite[Theorem~7.11, p.\,150]{ABR}.

\subsection{The sublinear growth Dirichlet problem}\label{section:SLG-BVP}

Given a Lebesgue measurable set $E\subseteq\mathbb{R}^n$ and $\theta\in[0,1)$ 
we define the space of sublinear growth functions of order $\theta$, 
denoted by ${\rm SLG}_\theta(E)$, as the collection of Lebesgue measurable functions 
$w:E\to{\mathbb{C}}$ satisfying 
\begin{equation}
\|w\|_{{\rm SLG}_\theta(E)}:={\rm{ess\,sup}}_{x\in E}\frac{|w(x)|}{1+|x|^\theta}<\infty.
\end{equation}
Hence, ${\rm SLG}_0(E)=L^\infty(E)$. Also, it clear from definitions that 
for each continuous function $u\in{\rm SLG}_\theta(\mathbb{R}^{n}_{+})$ we have
\begin{equation}\label{aerge}
\|u\|_{*,\rho}\leq\frac{1+\rho^\theta}{\rho}\|u\|_{{\rm SLG}_\theta(\mathbb{R}^{n}_{+})}
\,\,\text{ for each }\,\,\rho\in(0,\infty).
\end{equation}
As a consequence, any continuous function in ${\rm SLG}_\theta({\mathbb{R}}^n_{+})$ with $\theta\in[0,1)$ has 
subcritical growth, i.e.,
\begin{equation}\label{ahTva}
\lim_{\rho\to\infty}\|u\|_{*,\rho}=0\,\,\text{ for each }\,\,
u\in{\mathcal{C}}^0({\mathbb{R}}^n_{+})\cap{\rm SLG}_\theta({\mathbb{R}}^n_{+})\,\,\text{ with }\,\,\theta\in[0,1).
\end{equation}
In fact, for each continuous function $u:{\mathbb{R}}^n_{+}\to{\mathbb{C}}$ we have 
\begin{equation}\label{aerge.GDFS}
\|u\|_{{\rm SLG}_\theta(\mathbb{R}^{n}_{+})}=\sup_{\rho>0}\frac{\rho}{1+\rho^\theta}\|u\|_{*,\rho}.
\end{equation}
Indeed, the right-pointing inequality is clear from \eqref{aerge}, while the left-pointing inequality in 
\eqref{aerge.GDFS} may be justified by writing 
\begin{align}\label{aerge.GDFS.2}
\frac{|u(x)|}{1+|x|^\theta} &\leq\frac{|x|}{1+|x|^\theta}\Big(|x|^{-1}\cdot\sup_{B(0,|x|)}|u|\Big)
=\frac{|x|}{1+|x|^\theta}\|u\|_{*,|x|}
\nonumber\\[6pt]
&\leq\sup_{\rho>0}\frac{\rho}{1+\rho^\theta}\|u\|_{*,\rho}\,\,\text{ for each }\,\,x\in\mathbb{R}^{n}_{+},
\end{align}
and then taking the supremum over all $x\in\mathbb{R}^{n}_{+}$. Finally, we wish to note that 
\begin{equation}\label{ahTva.222}
{\rm SLG}_\theta({\mathbb{R}}^{n-1})\subset{\rm SCG}(\mathbb{R}^{n-1})
\,\,\text{ whenever }\,\,\theta\in[0,1).
\end{equation}

The following result from \cite{SCGC} extends Theorem~\ref{thm:Linfty} (which corresponds to the case when $\theta=0$).

\begin{theorem}\label{thm:SLG}
Let $L$ be an $M\times M$ homogeneous constant complex coefficient elliptic second-order system in ${\mathbb{R}}^n$, 
and fix an aperture parameter $\kappa>0$ along with some exponent $\theta\in[0,1)$. Then the sublinear growth Dirichlet 
boundary value problem for $L$ in $\mathbb{R}^{n}_{+}$, formulated as 
\begin{equation}\label{Dir-BVP-SLG}
\left\{
\begin{array}{l}
u\in\big[{\mathcal{C}}^\infty(\mathbb{R}^{n}_{+})\cap{\rm SLG}_\theta(\mathbb{R}^{n}_{+})\big]^M,
\\[4pt]
Lu=0\,\,\mbox{ in }\,\,\mathbb{R}^{n}_{+},
\\[8pt]
u\big|^{{}^{\kappa-{\rm n.t.}}}_{\partial{\mathbb{R}}^{n}_{+}}=f
\,\,\text{ at ${\mathcal{L}}^{n-1}$-a.e. point in }\,\,{\mathbb{R}}^{n-1},
\end{array}
\right.
\end{equation}
has a unique solution for each $f\in\big[{\rm SLG}_\theta(\mathbb{R}^{n-1})\big]^M$. 
Moreover, the solution $u$ of \eqref{Dir-BVP-SLG} is given by \eqref{exist:u} and satisfies 
\begin{equation}\label{q43t3gqdfr}
\|f\|_{[{\rm SLG}_\theta(\mathbb{R}^{n-1})]^M}\leq
\|u\|_{[{\rm SLG}_\theta(\mathbb{R}^{n}_{+})]^M}
\leq C\|f\|_{[{\rm SLG}_\theta(\mathbb{R}^{n-1})]^M}
\end{equation}
for some constant $C\in[1,\infty)$ depending only on $L$, $n$, and $\theta$.
\end{theorem}

Thanks to \eqref{ahTva.222} plus the fact that \eqref{q43t3g} and \eqref{aerge.GDFS} 
readily imply \eqref{q43t3gqdfr}, we may regard \eqref{Dir-BVP-SLG} 
as a ``sub-problem'' of \eqref{Dir-BVP-SCG}. Such a point of view then guarantees 
existence (in the class of solutions specified in \eqref{Dir-BVP-SLG}),
uniqueness, and also the estimate claimed in \eqref{q43t3gqdfr}.

The linear function $u(x',t)=ta$ for each $(x',t)\in{\mathbb{R}}^n_{+}$ (where $a\in{\mathbb{C}}^M\setminus\{0\}$
is a fixed vector) serves as a counterexample to the version of Theorem~\ref{thm:SLG} corresponding 
to $\theta=1$. Thus, restricting the exponent $\theta$ to $[0,1)$ is optimal.

As first noted in \cite{SCGC}, we also have a Fatou-type result in the context of functions with 
sublinear growth (extending the case $p=\infty$ of Theorem~\ref{TFac-gR} which corresponds to $\theta=0$). 
This reads as follows:

\begin{theorem}\label{y7tgv}
Let $L$ be an $M\times M$ homogeneous constant complex coefficient elliptic second-order system in ${\mathbb{R}}^n$ 
and fix an aperture parameter $\kappa>0$. Assume $u\in\big[\mathcal{C}^\infty(\mathbb{R}^n_{+})\big]^M$ satisfies 
$Lu=0$ in $\mathbb{R}_{+}^n$. If $u\in\big[{\rm SLG}_\theta(\mathbb{R}^{n}_{+})\big]^M$ for some $\theta\in[0,1)$ then
\begin{align}\label{Tafva.2222loolo}
\begin{array}{l}
\big(u\big|^{{}^{\kappa-{\rm n.t.}}}_{\partial{\mathbb{R}}^n_{+}}\big)(x')
\,\,\text{ exists at ${\mathcal{L}}^{n-1}$-a.e. point }\,\,x'\in{\mathbb{R}}^{n-1},
\\[10pt]
\displaystyle
u\big|^{{}^{\kappa-{\rm n.t.}}}_{\partial{\mathbb{R}}^n_{+}}\,\,\text{ belongs to }\,\,
\big[{\rm SLG}_\theta({\mathbb{R}}^{n-1})\big]^M\subset\Big[L^1\Big({\mathbb{R}}^{n-1}\,,\,\frac{dx'}{1+|x'|^{n}}\Big)\Big]^M,
\\[12pt]
u(x',t)=\Big(P^L_t\ast\big(u\big|^{{}^{\kappa-{\rm n.t.}}}_{\partial{\mathbb{R}}^n_{+}}\big)\Big)(x')
\,\,\text{ for each point }\,\,(x',t)\in{\mathbb{R}}^n_{+}.
\end{array}
\end{align}
As a consequence of this and Theorem~\ref{thm:SLG}, in the present setting the following version 
of the Maximum Principle holds:
\begin{equation}\label{Taf-UHN.ER.4}
\big\|u\big|^{{}^{\kappa-{\rm n.t.}}}_{\partial{\mathbb{R}}^n_{+}}\big\|_{[{\rm SLG}_\theta(\mathbb{R}^{n-1})]^M}
\approx\|u\|_{[{\rm SLG}_\theta(\mathbb{R}^{n}_{+})]^M}.
\end{equation}
\end{theorem}

Since thanks to \eqref{aerge} we have 
\begin{equation}\label{yt6h}
\int_1^\infty\|u\|_{*,t}\,\frac{dt}{t}\leq\|u\|_{[{\rm SLG}_\theta(\mathbb{R}^{n}_{+})]^M} 
\int_1^\infty\frac{1+t^\theta}{t^2}\,dt<\infty,
\end{equation}	
we may invoke Corollary~\ref{thm:Fatou-SCG} to conclude that the properties listed in \eqref{Tafva.222234t5w5} hold. 
It remains to check that the second item in \eqref{Tafva.2222loolo} holds, and this may be seen directly from definitions. 

Once again, the linear function $u(x',t)=ta$ for each $(x',t)\in{\mathbb{R}}^n_{+}$ 
(where $a\in{\mathbb{C}}^M\setminus\{0\}$ is a fixed vector) becomes a counterexample 
to the version of Theorem~\ref{y7tgv} corresponding to the end-point case $\theta=1$. 
As such, restricting the exponent $\theta$ to $[0,1)$ is sharp.

\subsection{The Dirichlet problem with boundary data in H\"older spaces}

Given $E\subset\mathbb{R}^m$ (for some $m\in{\mathbb{N}}$) and $\theta>0$, we define the 
homogeneous H\"older space of order $\theta$ on $E$, denoted by $\dot{\mathcal{C}}^\theta(E)$, 
as the collection of functions $w:E\to{\mathbb{C}}$ satisfying 
\begin{equation}\label{yTFVC.1}
\|w\|_{\dot{\mathcal{C}}^\theta(E)}:=\sup_{\substack{x,y\in E\\ x\not=y}}\frac{|w(x)-w(y)|}{|x-y|^\theta}<\infty.
\end{equation}
Also, define the inhomogeneous H\"older space of order $\theta$ on $E$ as
\begin{equation}\label{yTFVC.2}
{\mathcal{C}}^\theta(E):=\big\{w\in\dot{\mathcal{C}}^\theta(E):\,\sup_{E}|w|<\infty\big\},
\end{equation}
and set $\|w\|_{{\mathcal{C}}^\theta(E)}:=\|w\|_{\dot{\mathcal{C}}^\theta(E)}+\sup_{E}|w|$ for each 
$w\in{\mathcal{C}}^\theta(E)$. Clearly, 
\begin{equation}\label{6er3dd}
{\mathcal{C}}^\theta(E)\subseteq\dot{\mathcal{C}}^\theta(E)\subseteq{\rm SLG}_\theta(E)\,\,\text{ for each }\,\,\theta>0.
\end{equation}
In particular, together with \eqref{ahTva} this implies that any function in 
$\dot{\mathcal{C}}^\theta({\mathbb{R}}^n_{+})$ with $\theta\in(0,1)$ has subcritical growth.

The well-posedness of the $\dot{\mathcal{C}}^\theta$-Dirichlet problem was studied in 
\cite{BMO-MMMM} (see also \cite{Holder-MMM}). Here we follow the approach in \cite{SCGC} 
which uses item {\it (d)} in Theorem~\ref{thm:Poisson.II} and Theorem~\ref{y7tgv} 
to give an alternative, conceptually simpler proof. 

\begin{theorem}\label{theor:Holder}
Let $L$ be an $M\times M$ homogeneous constant complex coefficient elliptic second-order system 
in ${\mathbb{R}}^n$, and fix $\theta\in(0,1)$. Then the $\dot{\mathcal{C}}^\theta$-Dirichlet boundary 
value problem for $L$ in $\mathbb{R}^{n}_{+}$, formulated as
\begin{equation}\label{Dir-BVP-Holder}
\left\{
\begin{array}{l}
u\in\big[{\mathcal{C}}^\infty(\mathbb{R}^{n}_{+})\cap\dot{\mathcal{C}}^\theta(\overline{\mathbb{R}^{n}_{+}})\big]^M,
\\[4pt]
Lu=0\,\,\mbox{ in }\,\,\mathbb{R}^{n}_{+},
\\[6pt]
u\big|_{\partial{\mathbb{R}}^{n}_{+}}=f\,\,\text{ on }\,\,{\mathbb{R}}^{n-1},
\end{array}
\right.
\end{equation}
has a unique solution for each $f\in\big[\dot{\mathcal{C}}^\theta(\mathbb{R}^{n-1})\big]^M$. 
The solution $u$ of \eqref{Dir-BVP-Holder} is given by \eqref{exist:u} and there exists a 
constant $C=C(n,L,\theta)\in[1,\infty)$ with the property that 
\begin{equation}\label{Dir-BVP-BMO-Car-frac22}
\|f\|_{[\dot{\mathcal{C}}^\theta(\mathbb{R}^{n-1})]^M}\leq
\|u\|_{[\dot{\mathcal{C}}^\theta(\overline{{\mathbb{R}}^n_{+}})]^M}
\leq C\,\|f\|_{[\dot{\mathcal{C}}^\theta(\mathbb{R}^{n-1})]^M}.
\end{equation}
\end{theorem}

To prove existence, consider $p:=\big(1+\tfrac{\theta}{n-1}\big)^{-1}$ and 
note that this further implies $p\in\big(\tfrac{n-1}{n}\,,\,1\big)$ and 
$\theta=(n-1)\big(\tfrac{1}{p}-1\big)$.
In particular, with $\sim$ denoting the equivalence relation identifying any two 
functions which differ by a constant (cf., e.g., \cite[Theorem~5.30, p.307]{GCRF85}),
we have 
$\big(H^p(\mathbb{R}^{n-1})\big)^\ast=\dot{\mathcal{C}}^{\theta}({\mathbb{R}}^{n-1})\big/\sim$.
Next, given an arbitrary function 
$f=(f_\beta)_{1\leq\beta\leq M}\in\big[\dot{\mathcal{C}}^\theta(\mathbb{R}^{n-1})\big]^{M}$,
it is meaningful to define $u(x',t):=(P^L_t\ast f)(x')$ for all $(x',t)\in{\mathbb{R}}^n_{+}$. Then 
$u\in\big[{\mathcal{C}}^\infty(\mathbb{R}^{n}_{+})\big]^M$ satisfies $Lu=0$ in $\mathbb{R}^{n}_{+}$,
and for each $j\in\{1,\dots,n\}$ and each $(x',t)\in{\mathbb{R}}^n_{+}$ we have
\begin{align}\label{exist:u-1Rfa.iy}
t^{1-\theta}(\partial_j u)(x',t)=\Bigg\{\Big\langle t^{1-(n-1)(\frac{1}{p}-1)}(\partial_j K^L_{\alpha\beta})(x'-\cdot,t)\,,\,
[f_\beta]\Big\rangle\Bigg\}_{1\leq\alpha\leq M}
\end{align}
where $\langle\cdot,\cdot\rangle$ is the pairing between distributions belonging to the Hardy space 
$H^p(\mathbb{R}^{n-1})$ and equivalence classes {\rm (}modulo constants{\rm )} of functions belonging to 
$\dot{\mathcal{C}}^{\theta}({\mathbb{R}}^{n-1})$. In turn, based on \eqref{exist:u-1Rfa.iy}, for each 
$(x',t)\in{\mathbb{R}}^n_{+}$ and $j\in\{1,\dots,n\}$ we may estimate 
\begin{align}\label{exist:u-1Rfa.iy.2}
&\big|t^{1-\theta}(\partial_j u)(x',t)\big|
\\[6pt]
&\hskip 0.50in\leq 
\big\|t^{1-(n-1)(\frac{1}{p}-1)}(\partial_j K^L)(x'-\cdot,t)\big\|_{[H^p(\mathbb{R}^{n-1})]^{M\times M}}
\|f\|_{[\dot{\mathcal{C}}^\theta(\mathbb{R}^{n-1})]^{M}}.
\nonumber
\end{align}
In view of \eqref{grefr}, this further entails the existence of a constant $C\in(0,\infty)$ with the property that
\begin{align}\label{exist:u-1Rfa.iy.3}
\sup_{(x',t)\in{\mathbb{R}}^n_{+}}\Big\{t^{1-\theta}\big|(\nabla u)(x',t)\big|\Big\}
\leq C\|f\|_{[\dot{\mathcal{C}}^\theta(\mathbb{R}^{n-1})]^{M}}.
\end{align}
On the other hand, a well-known elementary argument (of a purely real-variable nature, 
based solely on the Mean-Value Theorem; see, e.g., \cite[\S6, Step~4]{BMO-MMMM}) implies that, 
for some constant $C=C(n,\theta)\in(0,\infty)$,  
\begin{align}\label{exist:u-1Rfa.iy.4}
\|u\|_{[\dot{\mathcal{C}}^\theta({\mathbb{R}}^n_{+})]^M}
\leq C\sup_{(x',t)\in{\mathbb{R}}^n_{+}}\Big\{t^{1-\theta}\big|(\nabla u)(x',t)\big|\Big\}.
\end{align}
At this stage, \eqref{Dir-BVP-BMO-Car-frac22} follows by combining \eqref{exist:u-1Rfa.iy.3} with 
\eqref{exist:u-1Rfa.iy.4}, keeping in mind the natural identification 
$\dot{\mathcal{C}}^\theta({\mathbb{R}}^n_{+})\equiv\dot{\mathcal{C}}^\theta(\overline{{\mathbb{R}}^n_{+}})$.
This finishes the proof of the existence for the problem \eqref{Dir-BVP-Holder}, 
and the justification of \eqref{Dir-BVP-BMO-Car-frac22}. In view of \eqref{6er3dd}, 
uniqueness for the problem \eqref{Dir-BVP-Holder} follows from Theorem~\ref{y7tgv}. 

\medskip 

As a byproduct of the above argument, we see that for each $\theta\in(0,1)$ we have
\begin{align}\label{exist:u-1Rfa.iy.4tR}
\begin{array}{c}
\|u\|_{[\dot{\mathcal{C}}^\theta(\overline{{\mathbb{R}}^n_{+}})]^M}
\approx\sup\limits_{(x',t)\in{\mathbb{R}}^n_{+}}\Big\{t^{1-\theta}\big|(\nabla u)(x',t)\big|\Big\}
\approx\big\|u\big|_{\partial{\mathbb{R}}^{n}_{+}}\big\|_{[\dot{\mathcal{C}}^\theta(\mathbb{R}^{n-1})]^M}
\\[12pt]
\text{uniformly for 
$u\in\big[{\mathcal{C}}^\infty(\mathbb{R}^{n}_{+})\cap\dot{\mathcal{C}}^\theta(\overline{\mathbb{R}^{n}_{+}})\big]^M$ 
with $Lu=0$ in $\mathbb{R}^{n}_{+}$}.
\end{array}
\end{align}
In this regard, let us also remark that for each exponent $\theta\in(0,1)$ 
and each aperture parameter $\kappa>0$ we also have
\begin{align}\label{exist:u-1Rfa.iy.4tR.2}
\begin{array}{c}
\|u\|_{[\dot{\mathcal{C}}^\theta(\overline{{\mathbb{R}}^n_{+}})]^M}
\approx\sup\limits_{x'\in{\mathbb{R}}^{n-1}}\|u\|_{[\dot{\mathcal{C}}^\theta(\Gamma_\kappa(x'))]^M}
\,\,\text{ uniformly for} 
\\[12pt]
u\in\big[{\mathcal{C}}^\infty(\mathbb{R}^{n}_{+})\cap\dot{\mathcal{C}}^\theta(\overline{\mathbb{R}^{n}_{+}})\big]^M
\,\,\text{ with }\,\,Lu=0\,\,\text{ in }\,\,\mathbb{R}^{n}_{+}.
\end{array}
\end{align}
To justify this, let the function $u$ be as in the last line above and set $f:=u\big|_{\partial{\mathbb{R}}^{n}_{+}}$. 
Then, given any $x',y'\in{\mathbb{R}}^{n-1}$, if $z$ is the point in 
$\overline{\Gamma_\kappa(x')}\cap\overline{\Gamma_\kappa(y')}$ closest to $\partial\mathbb{R}^{n}_{+}$ 
we may estimate 
\begin{align}\label{exist:u-1Rfa.iy.4tR.3}
|f(x')-f(y')| &\leq|u(x')-u(z)|+|u(y')-u(z)|
\nonumber\\[6pt]
&\leq\|u\|_{[\dot{\mathcal{C}}^\theta(\Gamma_\kappa(x'))]^M}|x'-z|^\theta
+\|u\|_{[\dot{\mathcal{C}}^\theta(\Gamma_\kappa(y'))]^M}|y'-z|^\theta
\nonumber\\[6pt]
&\leq C\Big(\sup\limits_{\xi\in{\mathbb{R}}^{n-1}}\|u\|_{[\dot{\mathcal{C}}^\theta(\Gamma_\kappa(\xi))]^M}\Big)|x'-y'|^\theta,
\end{align}
for some $C=C(\kappa,\theta)\in(0,\infty)$. Hence, 
\begin{align}\label{exist:u-1Rfa.iy.3amn}
\big\|u\big|_{\partial{\mathbb{R}}^{n}_{+}}\big\|_{[\dot{\mathcal{C}}^\theta(\mathbb{R}^{n-1})]^M}
=\|f\|_{[\dot{\mathcal{C}}^\theta(\mathbb{R}^{n-1})]^{M}}
\leq C\sup\limits_{x'\in{\mathbb{R}}^{n-1}}\|u\|_{[\dot{\mathcal{C}}^\theta(\Gamma_\kappa(x'))]^M} 
\end{align}
which, together with \eqref{exist:u-1Rfa.iy.4tR}, establishes the left-pointing inequality in the first 
line of \eqref{exist:u-1Rfa.iy.4tR.2}. Since the right-pointing inequality is trivial, this concludes 
the proof of \eqref{exist:u-1Rfa.iy.4tR.2}. 

\medskip 

As an immediate consequence of Theorem~\ref{theor:Holder} and Theorem~\ref{thm:Linfty} we obtain the 
following well-posedness result for the Dirichlet problem with boundary data from {\it inhomogeneous} 
H\"older spaces. 

\begin{corollary}\label{theor:Holder-IN}
Let $L$ be an $M\times M$ homogeneous constant complex coefficient elliptic second-order system 
in ${\mathbb{R}}^n$, and fix $\theta\in(0,1)$. Then the ${\mathcal{C}}^\theta$-Dirichlet boundary 
value problem for $L$ in $\mathbb{R}^{n}_{+}$, formulated as
\begin{equation}\label{Dir-BVP-Holder-IN}
\left\{
\begin{array}{l}
u\in\big[{\mathcal{C}}^\infty(\mathbb{R}^{n}_{+})\cap{\mathcal{C}}^\theta(\overline{\mathbb{R}^{n}_{+}})\big]^M,
\\[4pt]
Lu=0\,\,\mbox{ in }\,\,\mathbb{R}^{n}_{+},
\\[6pt]
u\big|_{\partial{\mathbb{R}}^{n}_{+}}=f\,\,\text{ on }\,\,{\mathbb{R}}^{n-1},
\end{array}
\right.
\end{equation}
has a unique solution for each $f\in\big[{\mathcal{C}}^\theta(\mathbb{R}^{n-1})\big]^M$. 
The solution $u$ of \eqref{Dir-BVP-Holder-IN} is given by \eqref{exist:u} and there exists a 
constant $C=C(n,L,\theta)\in[1,\infty)$ with the property that 
\begin{equation}\label{Dir-BVP-BMO-Car-frac22-IN}
\|f\|_{[{\mathcal{C}}^\theta(\mathbb{R}^{n-1})]^M}\leq
\|u\|_{[{\mathcal{C}}^\theta(\overline{{\mathbb{R}}^n_{+}})]^M}
\leq C\,\|f\|_{[{\mathcal{C}}^\theta(\mathbb{R}^{n-1})]^M}.
\end{equation}
\end{corollary}

As mentioned earlier in the narrative (cf. \eqref{TYd-YG-76g}), the Lam\'e system of elasticity 
fits into the general framework considered in this paper and, as such, all results so far 
apply to this special system. In this vein, it is of interest to raise the following issue:

\vskip 0.08in
{\bf Open Question~7.} 
{\it Formulate and prove Fatou-type theorems and well-posedness results for various versions of 
the Dirichlet problem in the upper half-space, of the sort discussed in this paper, for the Stokes 
system of hydrodynamics.}
\vskip 0.08in

\subsection{The Dirichlet problem with data in {\rm BMO} and {\rm VMO}}

In his ground breaking 1971 article \cite{Fe}, C.~Fefferman writes 
``{\it The main idea in proving {\rm [}that the dual of the Hardy space $H^1$ 
is the John-Nirenberg space ${\rm BMO}${\rm ]} is to study the ${\rm [}harmonic${\rm ]} 
Poisson integral of a function in ${\rm BMO}$.}'' For example, the key PDE result announced 
by C.~Fefferman in \cite{Fe} states that 
\begin{equation}\label{L-dJHG}
\displaystyle
\parbox{11.10cm}{a measurable function $f$ with
$\displaystyle\int_{{\mathbb{R}}^{n-1}}|f(x')|(1+|x'|)^{-n}\,dx'<+\infty$ belongs to the space
${\rm BMO}({\mathbb{R}}^{n-1})$ if and only if its Poisson integral 
$u:{\mathbb{R}}^n_{+}\to{\mathbb{R}}$, with respect to the Laplace operator
in ${\mathbb{R}}^n$, satisfies $\displaystyle\sup\limits_{x'\in{\mathbb{R}}^{n-1}}\sup\limits_{r>0}
\Big\{r^{1-n}\int\limits_{|x'-y'|<r}\int_0^r|(\nabla u)(y',t)|^2\,t\,dt\,dx'\Big\}<+\infty$.}
\end{equation}
One of the primary aims in \cite{BMO-MMMM} was to advance this line of research by developing machinery 
capable of dealing with the scenario in which the Laplacian in \eqref{L-dJHG} is replaced by much more 
general second-order elliptic systems with complex coefficients. To review the relevant results in this 
regard, some notation is needed. 

A Borel measure $\mu$ in $\mathbb{R}^{n}_{+}$ is said to be a Carleson measure in 
$\mathbb{R}^{n}_{+}$ provided
\begin{equation}\label{defi-Carleson}
\|\mu\|_{\mathcal{C}(\mathbb{R}_{+}^{n})}:=\sup_{Q\subset\mathbb{R}^{n-1}} 
\frac{1}{{\mathcal{L}}^{n-1}(Q)}\int_{0}^{\ell(Q)}\int_Q d\mu(x',t)<\infty,
\end{equation}
where the supremum runs over all cubes $Q$ in $\mathbb{R}^{n-1}$ (with sides parallel to the 
coordinate axes), and $\ell(Q)$ is the side-length of $Q$. Call a Borel measure 
$\mu$ in $\mathbb{R}^{n}_{+}$ a vanishing Carleson measure whenever
$\mu$ is a Carleson measure to begin with and, in addition,  
\begin{equation}\label{defi-CarlesonVan}
\lim_{r\to 0^{+}}\left(\sup_{Q\subset\mathbb{R}^{n-1},\,\ell(Q)\leq r} 
\frac{1}{{\mathcal{L}}^{n-1}(Q)}
\int_{0}^{\ell(Q)}\int_Q d\mu(x',t)\right)=0.
\end{equation}

Next, the Littlewood-Paley measure associated with a continuously differentiable function 
$u$ in ${\mathbb{R}}^n_{+}$ is $|\nabla u(x',t)|^2\,t\,dx'dt$, and we set  
\begin{equation}\label{ustarstar}
\|u\|_{**}:=\sup_{Q\subset\mathbb{R}^{n-1}}\left(\frac{1}{{\mathcal{L}}^{n-1}(Q)}
\int_{0}^{\ell(Q)}
\int_Q|\nabla u(x',t)|^2\,t\,dx'dt\right)^\frac12.
\end{equation}
In particular, for a continuously differentiable function $u$ in ${\mathbb{R}}^n_{+}$ we have 
\begin{equation}\label{ncud}
\|u\|_{**}<\infty\,\,\Longleftrightarrow\,\,
|\nabla u(x',t)|^2\,t\,dx'dt\,\,\text{ is a Carleson measure in }\,\,{\mathbb{R}}^n_{+}.
\end{equation}

The John-Nirenberg space $\mathrm{BMO}(\mathbb{R}^{n-1})$, of functions of bounded mean oscillations 
in ${\mathbb{R}}^{n-1}$, is defined as the collection of complex-valued functions 
$f\in L^1_{\rm loc}(\mathbb{R}^{n-1})$ satisfying  
\begin{equation}\label{defi-BMO}
\|f\|_{\mathrm{BMO}(\mathbb{R}^{n-1})}:=
\sup_{Q\subset\mathbb{R}^{n-1}}\frac{1}{{\mathcal{L}}^{n-1}(Q)}
\int_Q\big|f(x')-f_Q\big|\,dx'<\infty,
\end{equation}
where $f_Q:=\tfrac{1}{{\mathcal{L}}^{n-1}(Q)}\int_Qf\,d{\mathcal{L}}^{n-1}$ 
for each cube $Q$ in $\mathbb{R}^{n-1}$, and with the supremum taken over all such cubes $Q$. 
It turns out (cf., e.g., \cite{FS}) that 
\begin{equation}\label{eq:aaAabgr-22.aaa}
{\rm BMO}\big({\mathbb{R}}^{n-1}\big)\subset
L^1\Big({\mathbb{R}}^{n-1}\,,\,\frac{dx'}{1+|x'|^n}\Big)
\end{equation}
which opens the door for considering the convolution of the Poisson kernel from Theorem~\ref{thm:Poisson}
with {\rm BMO} functions in $\mathbb{R}^{n-1}$ (cf. item {\it (c)} in Theorem~\ref{thm:Poisson.II}).
Clearly, for every $f\in L^1_{\rm loc}(\mathbb{R}^{n-1})$ we have
\begin{equation}\label{defi-BMO-CCC}
\begin{array}{ll}
\|f\|_{\mathrm{BMO}(\mathbb{R}^{n-1})}=\|f+C\|_{\mathrm{BMO}(\mathbb{R}^{n-1})}, & \forall\,C\in{\mathbb{C}},
\\[6pt]
\|f\|_{\mathrm{BMO}(\mathbb{R}^{n-1})}=\|\tau_{z'}f\|_{\mathrm{BMO}(\mathbb{R}^{n-1})}, & \forall\,z'\in{\mathbb{R}}^{n-1},
\\[6pt]
\|f\|_{\mathrm{BMO}(\mathbb{R}^{n-1})}=\|\delta_{\lambda}f\|_{\mathrm{BMO}(\mathbb{R}^{n-1})}, & \forall\,\lambda\in(0,\infty),
\end{array}
\end{equation}
where $\tau_{z'}$ is the operator of translation by $z'$, i.e., 
$(\tau_{z'}f)(x'):=f(x'+z')$ for every $x'\in{\mathbb{R}}^{n-1}$, 
and $\delta_\lambda$ is the operator of dilation by $\lambda$, i.e., 
$(\delta_\lambda f)(x'):=f(\lambda x')$ for every $x'\in{\mathbb{R}}^{n-1}$.

As visible from the first line of \eqref{defi-BMO-CCC}, it happens that 
$\|\cdot\|_{\mathrm{BMO}(\mathbb{R}^{n-1}}$ is only a seminorm. Indeed, 
for every $f\in L^1_{\rm loc}(\mathbb{R}^{n-1})$ we have 
$\|f\|_{\mathrm{BMO}(\mathbb{R}^{n-1})}=0$ if and only if $f$ is a constant (in ${\mathbb{C}}$) 
at ${\mathcal{L}}^{n-1}$-a.e. in ${\mathbb{R}}^{n-1}$. Occasionally, we find it useful to mod out 
its null-space, in order to render the resulting quotient space Banach. Specifically, for two 
$\mathbb{C}$-valued Lebesgue measurable functions $f,g$ defined in $\mathbb{R}^{n-1}$ we say that $f\sim g$ 
provided $f-g$ is constant ${\mathcal{L}}^{n-1}$-a.e. in ${\mathbb{R}}^{n-1}$. 
This is an equivalence relation and we let 
\begin{align}\label{jgsyjw-AASSS}
[f]:=\big\{g:\mathbb{R}^{n-1}\to\mathbb{C}:\,\text{$g$ measurable and $f\sim g$}\big\}
\end{align}
denote the equivalence class of any given $\mathbb{C}$-valued Lebesgue measurable function $f$ 
defined in $\mathbb{R}^{n-1}$. In particular, the quotient space 
\begin{equation}\label{defi-BMO-tilde}
\widetilde{\mathrm{BMO}}(\mathbb{R}^{n-1}):=\big\{[f]:\,f\in{\mathrm{BMO}}(\mathbb{R}^{n-1})\big\}.
\end{equation}
becomes complete (hence Banach) when equipped with the norm
\begin{align}\label{defi-BMO-nbgxcr}
\big\|\,[f]\,\big\|_{\widetilde{\mathrm{BMO}}(\mathbb{R}^{n-1})}:=\|f\|_{\mathrm{BMO}(\mathbb{R}^{n-1})}
\,\,\text{ for each }\,\,f\in{\mathrm{BMO}}(\mathbb{R}^{n-1}).
\end{align}

Moving on, the Sarason space of $\mathbb{C}$-valued functions of vanishing 
mean oscillations in ${\mathbb{R}}^{n-1}$ is defined by
\begin{align}\label{defi-VMO}
{\mathrm{VMO}}(\mathbb{R}^{n-1})&:=\Bigg\{f\in{\mathrm{BMO}}(\mathbb{R}^{n-1}):
\\[-6pt]
&\hskip 0.30in
\lim_{r\to 0^{+}}\left(\sup_{Q\subset\mathbb{R}^{n-1},\,\ell(Q)\leq r}\,\,
\frac{1}{{\mathcal{L}}^{n-1}(Q)}\int_Q\big|f(x')-f_Q\big|\,dx'\right)=0\Bigg\}.
\nonumber
\end{align}
The space ${\mathrm{VMO}}(\mathbb{R}^{n-1})$ turns out to be a closed subspace of 
${\mathrm{BMO}}(\mathbb{R}^{n-1})$. In fact, if ${\mathrm{UC}}({\mathbb{R}}^{n-1})$ stands for the space
of $\mathbb{C}$-valued uniformly continuous functions in ${\mathbb{R}}^{n-1}$, then a well-known result of 
Sarason \cite[Theorem~1, p.\,392]{Sa75} implies that, in fact, 
\begin{equation}\label{ku6ffcfc}
\parbox{10.70cm}{$f\in{\mathrm{BMO}}({\mathbb{R}}^{n-1})$ belongs 
to the space ${\mathrm{VMO}}({\mathbb{R}}^{n-1})$ if and only if there exists a sequence 
$\{f_j\}_{j\in{\mathbb{N}}}\subset{\mathrm{UC}}({\mathbb{R}}^{n-1})\cap{\mathrm{BMO}}({\mathbb{R}}^{n-1})$ 
such that $\|f-f_j\|_{{\mathrm{BMO}}({\mathbb{R}}^{n-1})}\longrightarrow 0$ as $j\to\infty$.}
\end{equation}
Another characterization of ${\mathrm{VMO}}(\mathbb{R}^{n-1})$
due to Sarason (cf. \cite[Theorem~1, p.\,392]{Sa75}) is as follows:
\begin{equation}\label{defi-VMO-SSS}
\parbox{10.70cm}{a given function $f\in{\mathrm{BMO}}(\mathbb{R}^{n-1})$ actually belongs to the space
${\mathrm{VMO}}(\mathbb{R}^{n-1})$ if and only if $\lim\limits_{{\mathbb{R}}^{n-1}\ni z'\to 0'}
\|\tau_{z'}f-f\|_{{\mathrm{BMO}}(\mathbb{R}^{n-1})}=0$.}
\end{equation}

We are now ready to recall the first main result from \cite{BMO-MMMM}. This concerns the well-posedness of the 
$\mathrm{BMO}$-Dirichlet problem in the upper half-space for systems $L$ as in 
\eqref{L-def}-\eqref{L-ell.X}. The existence of a unique solution is established 
in the class of functions $u$ satisfying a Carleson measure condition 
(expressed in terms of the finiteness of \eqref{ustarstar}). The formulation of 
the theorem emphasizes the fact that this contains as a ``sub-problem'' the 
$\mathrm{VMO}$-Dirichlet problem for $L$ in ${\mathbb{R}}^n_{+}$
(in which scenario $u$ satisfies a vanishing Carleson measure condition). 

\begin{theorem}\label{them:BMO-Dir}
Let $L$ be an $M\times M$ elliptic constant complex coefficient system as in 
\eqref{L-def}-\eqref{L-ell.X}, and fix an aperture parameter $\kappa>0$. 
Then the $\mathrm{BMO}$-Dirichlet boundary value problem for $L$ in $\mathbb{R}^{n}_{+}$, namely 
\begin{equation}\label{Dir-BVP-BMO}
\left\{
\begin{array}{l}
u\in\big[{\mathcal{C}}^\infty(\mathbb{R}^{n}_{+})\big]^M,\quad Lu=0\,\,\mbox{ in }\,\,\mathbb{R}^{n}_{+},
\\[4pt]
\big|\nabla u(x',t)\big|^2\,t\,dx'dt\,\,\mbox{is a Carleson measure in }\mathbb{R}^{n}_{+},
\\[6pt]
u\big|^{{}^{\kappa-{\rm n.t.}}}_{\partial{\mathbb{R}}^{n}_{+}}=f
\,\,\text{ at ${\mathcal{L}}^{n-1}$-a.e. point in }\,\,{\mathbb{R}}^{n-1},
\end{array}
\right.
\end{equation}
has a unique solution for each $f\in\big[\mathrm{BMO}(\mathbb{R}^{n-1})\big]^M$. 
Moreover, this unique solution satisfies the following additional properties:
\begin{list}{$(\theenumi)$}{\usecounter{enumi}\leftmargin=.8cm
\labelwidth=.8cm\itemsep=0.2cm\topsep=.1cm
\renewcommand{\theenumi}{\alph{enumi}}}
\item[(i)] With $P^L$ denoting the Poisson kernel for $L$ in $\mathbb{R}^{n}_{+}$ from
Theorem~\ref{thm:Poisson}, one has the Poisson integral representation formula
\begin{equation}\label{eqn-Dir-BMO:u}
u(x',t)=(P_t^L*f)(x'),\qquad\forall\,(x',t)\in{\mathbb{R}}^n_{+}.
\end{equation}
\item[(ii)] The size of the solution is comparable to the size of the boundary datum, i.e., there 
exists $C=C(n,L)\in(1,\infty)$ with the property that 
\begin{equation}\label{Dir-BVP-BMO-Car}
C^{-1}\|f\|_{[\mathrm{BMO}(\mathbb{R}^{n-1})]^M}\leq
\|u\|_{**}\leq C\,\|f\|_{[\mathrm{BMO}(\mathbb{R}^{n-1})]^M}.
\end{equation}
\item[(iii)] There exists a constant $C=C(n,L)\in(0,\infty)$ independent of $u$ 
with the property that the following uniform {\rm BMO} estimate holds:
\begin{equation}\label{feps-BTTGB}
\sup_{\varepsilon>0}\|u(\cdot,\varepsilon)\|_{[\mathrm{BMO}(\mathbb{R}^{n-1})]^M}
\leq C\,\|u\|_{**}.
\end{equation}
Moreover, $u$ satisfies a vanishing Carleson measure condition in $\mathbb{R}^{n}_{+}$ 
if and only if $u$ converges to its boundary datum vertically in 
$\big[\mathrm{BMO}(\mathbb{R}^{n-1})\big]^M$, i.e.,  
\begin{equation}\label{eqn:conv-Bfed}
{}\hskip 0.30in
\lim_{\varepsilon\to 0^+}\|u(\cdot,\varepsilon)-f\|_{[\mathrm{BMO}(\mathbb{R}^{n-1})]^M}=0
\Longleftrightarrow
\left\{
\begin{array}{l}
\big|\nabla u(x',t)\big|^2\,t\,dx'dt\,\,\,\text{is}
\\[4pt]
\text{a vanishing Carleson}
\\[4pt]
\text{measure in }\,\,\mathbb{R}^{n}_{+}.
\end{array}
\right.
\end{equation}
\item[(iv)] The following regularity results hold:
\begin{align}\label{Dir-BVP-Reg}
f\in\big[\mathrm{VMO}(\mathbb{R}^{n-1})\big]^M
& \Longleftrightarrow
\left\{
\begin{array}{l}
\big|\nabla u(x',t)\big|^2\,t\,dx'dt\,\,\mbox{is a vanishing}
\\[4pt]
\text{Carleson measure in }\,\,\mathbb{R}^{n}_{+}
\end{array}
\right.
\\[6pt]
& \Longleftrightarrow
\lim_{{\mathbb{R}}^n_{+}\ni z\to 0}\|\tau_z u-u\|_{**}=0,
\label{Dir-BVP-Reg.TTT}
\end{align}
where $(\tau_z u)(x):=u(x+z)$ for each $x,z\in{\mathbb{R}}^n_{+}$.
\end{list}

As a consequence, the $\mathrm{VMO}$-Dirichlet boundary value problem for $L$ 
in $\mathbb{R}^{n}_{+}$, i.e.,
\begin{equation}\label{Dir-BVP-VMO}
\left\{
\begin{array}{l}
u\in\big[{\mathcal{C}}^\infty(\mathbb{R}^{n}_{+})\big]^M,\quad Lu=0\,\,\mbox{ in }\,\,\mathbb{R}^{n}_{+},
\\[4pt]
\big|\nabla u(x',t)\big|^2\,t\,dx'dt\,\,\mbox{is a vanishing Carleson measure in }\mathbb{R}^{n}_{+},
\\[6pt]
u\big|^{{}^{\kappa-{\rm n.t.}}}_{\partial{\mathbb{R}}^{n}_{+}}=f
\,\,\text{ at ${\mathcal{L}}^{n-1}$-a.e. point in }\,\,{\mathbb{R}}^{n-1},
\end{array}
\right.
\end{equation}
has a unique solution for each $f\in\big[\mathrm{VMO}(\mathbb{R}^{n-1})\big]^M$. 
Moreover, its solution is given by \eqref{eqn-Dir-BMO:u}, satisfies \eqref{Dir-BVP-BMO-Car}-\eqref{feps-BTTGB}, and
\begin{equation}\label{eqn:conv-BfEE}
\lim_{\varepsilon\to 0^+}\|u(\cdot,\varepsilon)-f\|_{[\mathrm{BMO}(\mathbb{R}^{n-1})]^M}=0.
\end{equation}
\end{theorem}

It is reassuring to remark that replacing the original boundary datum $f$ by $f+C$ where $C\in{\mathbb{C}}^M$ 
in \eqref{Dir-BVP-BMO} changes the solution $u$ into $u+C$ (given that convolution with the Poisson 
kernel reproduces constants from ${\mathbb{C}}^M$; cf. \eqref{eq:IG6gy.2}). As such, the 
$\widetilde{\rm BMO}$-Dirichlet problem for $L$ in ${\mathbb{R}}^n_{+}$ is also well-posed, 
if uniqueness of the solution is now understood modulo constants from ${\mathbb{C}}^M$. 

\medskip 

The proof of Theorem~\ref{them:BMO-Dir} given in \cite{BMO-MMMM} employs a quantitative Fatou-type theorem, 
which includes a Poisson integral representation formula along with a characterization of {\rm BMO} in terms 
of boundary traces of null-solutions of elliptic systems in ${\mathbb{R}}^n_{+}$. A concrete statement is given 
below in Theorem~\ref{thm:fatou-ADEEDE}. Among other things, the said theorem shows that the demands formulated  
in the first two lines of \eqref{Dir-BVP-BMO} imply that the pointwise nontangential limit considered in 
the third line of \eqref{Dir-BVP-BMO} is always meaningful, and that the boundary datum should necessarily be 
selected from the space ${\rm BMO}$. This theorem also highlights the fact that it is natural to seek a 
solution of the $\mathrm{BMO}$ Dirichlet problem by taking the convolution of the boundary datum 
with the Poisson kernel $P^L$ associated with the system $L$. Finally, Theorem~\ref{thm:fatou-ADEEDE} 
readily implies the uniqueness of solution for the $\mathrm{BMO}$-Dirichlet problem \eqref{Dir-BVP-BMO}.

\begin{theorem}\label{thm:fatou-ADEEDE}
Let $L$ be an $M\times M$ elliptic system with constant complex coefficients as in \eqref{L-def}-\eqref{L-ell.X} 
and consider $P^L$, the Poisson kernel in $\mathbb{R}^{n}_{+}$ associated with $L$ as in Theorem~\ref{thm:Poisson}. 
Also, fix an aperture parameter $\kappa>0$. Then there exists a constant $C=C(L,n,\kappa)\in(1,\infty)$ with 
the property that 
\begin{eqnarray}\label{Tafva.BMO}
&&
\left\{
\begin{array}{r}
u\in\big[{\mathcal{C}}^\infty({\mathbb{R}}^n_{+})\big]^M
\\[4pt]
Lu=0\,\mbox{ in }\,{\mathbb{R}}^n_{+}
\\[6pt]
\text{and }\,\,\|u\|_{**}<\infty
\end{array}
\right.
\\[4pt]
&&\hskip 0.30in
\Longrightarrow
\left\{
\begin{array}{l}
u\big|^{{}^{\kappa-{\rm n.t.}}}_{\partial{\mathbb{R}}^n_{+}}\,\mbox{ exists a.e.~in }\,
{\mathbb{R}}^{n-1},\,\mbox{ lies in }\,\big[\mathrm{BMO}(\mathbb{R}^{n-1})\big]^M,
\\[12pt]
u(x',t)=\Big(P^L_t\ast\big(u\big|^{{}^{\kappa-{\rm n.t.}}}_{\partial{\mathbb{R}}^n_{+}}\big)\Big)(x')
\,\text{ for all }\,(x',t)\in{\mathbb{R}}^n_{+},
\\[12pt]
\mbox{and }\,C^{-1}\|u\|_{**}\leq
\big\|u\big|^{{}^{\kappa-{\rm n.t.}}}_{\partial{\mathbb{R}}^n_{+}}\big\|_{[\mathrm{BMO}(\mathbb{R}^{n-1})]^M}\leq C\|u\|_{**}.
\end{array}
\right.
\nonumber
\end{eqnarray}
In fact, the following characterization of $\big[\mathrm{BMO}(\mathbb{R}^{n-1})\big]^M$, adapted to the system $L$, 
holds:
\begin{equation}\label{eq:tr-sols}
\big[\mathrm{BMO}(\mathbb{R}^{n-1})\big]^M=\Big\{u\big|^{{}^{\kappa-{\rm n.t.}}}_{\partial{\mathbb{R}}^{n}_{+}}: 
u\in\big[\mathcal{C}^\infty(\mathbb{R}^n_{+})\big]^M,\,\,Lu=0\,\mbox{ in }\,\mathbb{R}^{n}_{+},\,\,\|u\|_{**}<\infty\Big\}.
\end{equation}

Moreover, 
\begin{equation}\label{eq:tr-OP-SP}
{\mathrm{LMO}}({\mathbb{R}}^n_{+}):=
\Big\{u\in\big[\mathcal{C}^\infty(\mathbb{R}^n_{+})\big]^M:\,
Lu=0\mbox{ in }\mathbb{R}^{n}_{+}\,\,\text{ and }\,\,\|u\|_{**}<\infty\Big\}
\end{equation}
is a linear space on which $\|\cdot\|_{**}$ is a seminorm with null-space
${\mathbb{C}}^M$, the quotient space ${\mathrm{LMO}}({\mathbb{R}}^n_{+})\big/{\mathbb{C}}^M$ 
becomes complete {\rm (}hence Banach{\rm )} when equipped with $\|\cdot\|_{**}$, and 
the nontangential pointwise trace operator acting on equivalence classes in the context 
\begin{equation}\label{eq:tr-OP}
{\mathrm{LMO}}({\mathbb{R}}^n_{+})\big/{\mathbb{C}}^M\ni[u]\longmapsto 
\big[u\big|^{{}^{\kappa-{\rm n.t.}}}_{\partial{\mathbb{R}}^{n}_{+}}\big]
\in\big[\widetilde{\mathrm{BMO}}(\mathbb{R}^{n-1})\big]^M
\end{equation}
is a well-defined linear isomorphism between Banach spaces, where
$[u]$ in \eqref{eq:tr-OP} denotes the equivalence class of $u$ in 
${\mathrm{LMO}}({\mathbb{R}}^n_{+})\big/{\mathbb{C}}^M$ and 
$\big[u\big|^{{}^{\kappa-{\rm n.t.}}}_{\partial{\mathbb{R}}^{n}_{+}}\big]$ is interpreted as
in \eqref{jgsyjw-AASSS}.
\end{theorem}

There is also a counterpart of the Fatou-type result stated as Theorem~\ref{thm:fatou-ADEEDE}
emphasizing the space {\rm VMO} in place of {\rm BMO}. Specifically, the following 
theorem was proved in \cite{BMO-MMMM}.

\begin{theorem}\label{thm:fatou-VMO}
Let $L$ be an $M\times M$ elliptic system with constant complex coefficients as in
\eqref{L-def}-\eqref{L-ell.X} and consider $P^L$, the associated Poisson kernel for 
$L$ in $\mathbb{R}^{n}_{+}$ from Theorem~\ref{thm:Poisson}. Also, fix an aperture parameter $\kappa>0$.
Then for any function 
\begin{equation}\label{Dir-BVP-VMOq1}
\text{$u\in\big[{\mathcal{C}}^\infty({\mathbb{R}}^n_{+})\big]^M$ satisfying 
$Lu=0$ in ${\mathbb{R}}^n_{+}$ and $\|u\|_{**}<\infty$}
\end{equation}
one has
\begin{equation}\label{Dir-BVP-VMOq2}
\left.
\begin{array}{r}
\big|\nabla u(x',t)\big|^2\,t\,dx'dt\,\,\mbox{is}
\\[4pt]
\text{a vanishing Carleson}
\\[4pt]
\text{measure in }\,\,\mathbb{R}^{n}_{+}
\end{array}
\right\}
\Longrightarrow
\left\{
\begin{array}{l}
u\big|^{{}^{\kappa-{\rm n.t.}}}_{\partial{\mathbb{R}}^n_{+}}\,\mbox{ exists a.e.~in }\,
{\mathbb{R}}^{n-1},\,\text{ and}
\\[12pt]
u\big|^{{}^{\kappa-{\rm n.t.}}}_{\partial{\mathbb{R}}^n_{+}}
\,\,\text{ is in }\,\,\big[{\mathrm{VMO}(\mathbb{R}^{n-1})\big]^M}.
\end{array}
\right.
\end{equation}

Furthermore, the following characterization of the space 
$\big[\mathrm{VMO}(\mathbb{R}^{n-1})\big]^M$, adapted to the system $L$, holds:
\begin{align}\label{eq:tr-sols-VMO}
\big[\mathrm{VMO}(\mathbb{R}^{n-1})\big]^M
&=\Big\{u\big|^{{}^{\kappa-{\rm n.t.}}}_{\partial{\mathbb{R}}^{n}_{+}}
:\,\, u\in\mathrm{LMO}(\mathbb{R}^n_{+})\,\,\text{ and }\,\,
\big|\nabla u(x',t)\big|^2\,t\,dx'dt
\nonumber\\[0pt] 
&\qquad
\text{is a vanishing Carleson measure in }\,\,\mathbb{R}^{n}_{+}\Big\}.
\end{align}
\end{theorem}

There is yet another version of the space of functions of vanishing mean oscillations which we would like to recall. 
To set the stage, let ${\mathcal{C}}^0_0({\mathbb{R}}^{n-1})$ be the space of all continuous functions in ${\mathbb{R}}^{n-1}$ 
which vanish at infinity, equipped with the supremum norm. Also, let $\{R_j\}_{1\leq j\leq n-1}$ be the family of Riesz
transforms in ${\mathbb{R}}^{n-1}$. Define ${\rm CMO}({\mathbb{R}}^{n-1})$ as the 
collection of all functions $f\in L^1_{\rm loc}({\mathbb{R}}^{n-1})$ which may be expressed as
\begin{equation}\label{utggG-TRFF.rt.A}
\begin{array}{c}
f=f_0+\sum_{j=1}^{n-1}R_jf_j\,\,\text{ in }\,\,{\mathbb{R}}^{n-1}\,\,\text{ in }\,\,{\mathbb{R}}^{n-1}
\\[6pt]
\text{for some }\,\,f_0,f_1,\dots,f_{n-1}\in{\mathcal{C}}^0_0({\mathbb{R}}^{n-1}),
\end{array}
\end{equation}
and set  
\begin{equation}\label{utggG-TRFF.rt.B}
\|f\|_{{\rm CMO}({\mathbb{R}}^{n-1})}:=\inf\Big\{\|f_0\|_{L^{\infty}({\mathbb{R}}^{n-1})}
+\sum_{j=1}^{n-1}\|f_j\|_{L^{\infty}({\mathbb{R}}^{n-1})}\Big\}
\end{equation}
where the infimum is taken over all possible representations of $f$ as in \eqref{utggG-TRFF.rt.A}.
Then ${\rm CMO}({\mathbb{R}}^{n-1})$ becomes a Banach space, which may be alternatively characterized 
as the pre-dual of the Hardy space $H^1({\mathbb{R}}^{n-1})$ (cf. \cite[(2.0'), p.\,185]{Neri}; see also 
\cite{Bourdaud} and \cite{CoWe77} for more on this topic). One may also show that ${\rm CMO}({\mathbb{R}}^{n-1})$ 
is a closed subspace of ${\rm BMO}({\mathbb{R}}^{n-1})$, and ${\mathcal{C}}^0_0({\mathbb{R}}^{n-1})$ is dense 
in ${\rm CMO}({\mathbb{R}}^{n-1})$. Hence, 
\begin{equation}\label{utggG-TRFF.rt.C}
\text{${\rm CMO}({\mathbb{R}}^{n-1})$ is the closure of ${\mathcal{C}}^0_0({\mathbb{R}}^{n-1})$ in 
${\rm BMO}({\mathbb{R}}^{n-1})$.}
\end{equation}
However, the Sarason space ${\rm VMO}({\mathbb{R}}^{n-1})$ (from \eqref{defi-VMO}) is strictly larger than 
${\rm CMO}({\mathbb{R}}^{n-1})$. In relation to the latter version of the space of functions of vanishing 
mean oscillations we wish to pose the following question. 

\vskip 0.08in
{\bf Open Question~8.} 
{\it Formulate and prove a well-posedness result for the Dirichlet problem in the upper half-space, 
for an $M\times M$ elliptic second-order homogeneous constant complex coefficient system $L$, with boundary 
data from $\big[{\rm CMO}({\mathbb{R}}^{n-1})\big]^M$. Also, prove a Fatou-type theorem for null-solutions 
of $L$ in ${\mathbb{R}}^n_{+}$, which naturally accompanies the said well-posedness result.}
\vskip 0.08in
To address these issues, a new brand of Carleson measure must be identified. 

We close by recording the following result proved in \cite{BMO-MMMM}. The first item can be thought of 
as an analogue of Fefferman's theorem, characterizing {\rm BMO} as in \eqref{L-dJHG}, in the case of 
elliptic systems with complex coefficients. The second item may be viewed as a characterization 
of {\rm VMO} in the spirit of Fefferman's original result. 

\begin{theorem}\label{thm:FEFF}
Let $L$ be an $M\times M$ elliptic system with constant complex coefficients as in
\eqref{L-def}-\eqref{L-ell.X} and consider the Poisson kernel $P^L$ in $\mathbb{R}^{n}_{+}$ associated  
with the system $L$ as in Theorem~\ref{thm:Poisson}. Also, assume $f:\mathbb{R}^{n-1}\to\mathbb{C}^{M}$ 
is a Lebesgue measurable function satisfying 
\begin{equation}\label{Di-AK}
\int_{{\mathbb{R}}^{n-1}}\frac{|f(x')|}{1+|x'|^{n}}\,dx'<\infty.
\end{equation}
Finally, let $u$ be the Poisson integral of $f$ in ${\mathbb{R}}^n_{+}$ with respect to the system $L$, i.e., 
$u:{\mathbb{R}}^n_{+}\to{\mathbb{C}}^M$ is given by $u(x',t):=(P^L_t\ast f)(x')$ 
for each $(x',t)\in{\mathbb{R}}^n_{+}$. Then the following statements are true.

\begin{list}{(\theenumi)}{\usecounter{enumi}\leftmargin=.8cm
		\labelwidth=.8cm\itemsep=0.2cm\topsep=.1cm
		\renewcommand{\theenumi}{\alph{enumi}}}

\item The function $f$ belongs to the space $\big[{\rm BMO}({\mathbb{R}}^{n-1})\big]^M$ 
if and only if $|\nabla u(x',t)|^2\,t\,dx'dt$ is a Carleson measure in 
${\mathbb{R}}^n_{+}$ {\rm (}or, equivalently, \\
$\|u\|_{\ast\ast}<\infty${\rm ;} cf. \eqref{ncud}{\rm )}.
\vskip 0.08in
\item The function $f$ belongs to the space $\big[{\rm VMO}({\mathbb{R}}^{n-1})\big]^M$ 
if and only if $|\nabla u(x',t)|^2\,t\,dx'dt$ is a vanishing Carleson measure in 
${\mathbb{R}}^n_{+}$.
\end{list}
\end{theorem}

\end{document}